\documentclass[12pt]{article}

\usepackage{amsfonts,amsmath,amsthm,latexsym,amssymb,fullpage,amscd}
\usepackage[mathscr]{eucal}
\usepackage{tikz}

 at 16 true pt
 at 13  true pt
 at 12 true pt

\newcommand{\R}{\mathbb{R}}
\newtheorem{theorem}{Theorem}[section]
\newtheorem{lemma}[theorem]{Lemma}

\title{A method for bounding oscillatory integrals in terms of non-oscillatory integrals}

\author{Michael Greenblatt}
\date{\today}

\newcommand\blfootnote[1]{%
  \begingroup
  \renewcommand\thefootnote{}\footnote{#1}%
  \addtocounter{footnote}{-1}%
  \endgroup
}
\begin{document}
\maketitle
\begin{abstract} 

We describe an elementary method for bounding a one-dimensional oscillatory integral in terms of an associated non-oscillatory integral. The bounds obtained
are efficient in an appropriate sense and behave well under perturbations of the phase. As a consequence, for an $n$-dimensional oscillatory integral with a critical point at the origin, 
we may apply the one-dimensional
estimates in the radial direction and then integrate the result, thereby obtaining
natural bounds for the $n$-dimensional oscillatory integral in terms of the measures of the sublevel sets associated with the phase. To illustrate, we provide several classes of examples,
 including situations where the phase function has a critical point at which it vanishes to infinite order. 

\end{abstract}
\blfootnote{This work was supported by a grant from the Simons Foundation.}

\section{Background and theorem statements.}

\subsection{General discussion and the main theorems.}

In a number of settings in analysis, one considers  oscillatory integrals of the form
\[I(\lambda) = \int  e^{i\lambda f(x_1,...,x_n)} \phi(x_1,...,x_n)\,dx_1\,...\,dx_n\tag{1.1}\]
Here,  $\phi(x_1,...,x_n)$ is a compactly supported real-valued $C^1$ function defined on a neighborhood of the origin and 
$f(x_1,...,x_n)$ is at least a $C^2$ real-valued function on an open set containing the support of $\phi$. The goal here is to find estimates of the form $|I(\lambda)| \leq g(|\lambda|)$, where
$g$ is an appropriately quickly decreasing function. To avoid trivialities, one normally assumes that $f$ is nonconstant and $\nabla f(0) = 0$. 

A canonical example of where oscillatory integrals $(1.1)$ show up is in the analysis of Fourier transforms of surface measures. If one is looking at the Fourier transform in a specific direction of a smooth surface
measure for a surface $S$, then after a rotation so that the direction becomes $(0,...,0,1)$, the resulting oscillatory integral is locally exactly of the form $(1.1)$. Here
 $f(x)$ is the function for which $S$ is the graph of $f(x)$. Since surface measure Fourier transforms have applications to a wide variety of subjects including maximal averages, restriction
problems, lattice point discrepancy, and more, improved understanding of oscillatory integrals $(1.1)$  leads to developments in those subjects.

One often estimates a one-dimensional oscillatory integral by applying the method of stationary phase or its consequences, for example
 using the Van der Corput lemma. The Van der Corput 
lemma alone may not give optimal estimates when averaging over the remaining $n-1$ variables. A more refined use of stationary phase may improve 
things, but such methods can give endpoint terms and sublevel set measures  that may be hard to average effectively over these remaining variables. 

In this paper, we will give a theorem (Theorem 1.1) that bounds a one-dimensional oscillatory integral in terms of a single non-oscillatory integral, combined with a single endpoint term.
This can be used in the oscillatory integral $(1.1)$ by integrating in the radial direction, applying Theorem 1.1, and then integrating the result over the unit sphere. This "oscillatory 
integral method of rotations" 
frequently bounds the oscillatory integral by associated sublevel set measures in a natural way, as we will see in section 3, as the maximal measure of a the set where the phase function 
is within a single period. 

To prove the uniformity of the relevant constants in Theorem 1.1 over the unit sphere for real analytic $f(x)$, we
 will make use of a result (Theorem 2.1) providing uniform bounds above and below on parallel lines for real analytic functions and/or its derivatives.
This will enable us to average Theorem 1.1 in the above manner. This result leads to some corollaries (Corollaries 
2.1.1-2.1.3) concerning the behavior of real analytic functions on parallel lines which may be of interest to the reader in their own right. Theorem 2.1 will be proven in section 6.

We also prove a theorem (Theorem 1.2) with weaker hypotheses which gives a similar statement to Theorem 1.1 except it includes some endpoint terms
that are not stable under general perturbations of the phase. Nonetheless, this result provides useful estimates for various classes of phase functions. Similarly to the 
real analytic case, in section 3 we will use this result in the radial direction and then integrate over the sphere to bound oscillatory integrals by sublevel set measures in a natural way,
 this time for classes of phases with zeroes of infinite order at the origin. 

In section 4, we will describe a few ways in which Theorems 1.1 and 1.2 can be viewed as efficient, and in section 5 we will provide the proofs of Theorems 1.1 and 1.2.

We should point out there have been various papers focusing on scalar oscillatory integrals of the type considered here. We mention
 [BaGuZhZo] [CaCWr] [G] [Gi] [Gr] [PhStS] [V] as some examples.

Throughout this paper, for a $C^k$ function $f(x)$ on a closed interval $[a,b]$, for $p \leq k$ we will write $f^{(p)}(x)$ to denote the standard
 $p$th derivative if $x \in (a,b)$, and the appropriate left or right-hand $p$th derivative if $x$ is an endpoint $a$ or $b$.

To give an idea of what we will be doing, we consider the simple one dimensional integral $\int_a^b e^{i f(x)}\,dx$ for $f \in C^2([a,b])$.
Let $[c,d] \subset [a,b]$ such that  $f'(x) \neq 0$ for $x \in [c,d]$. We consider the portion of the
 integral over $[c,d]$. We apply the standard integration by parts, namely writing $ e^{i f(x)} = if'(x)e^{i f(x)} 
\times {1 \over if'(x)}$ and integrating by parts, integrating $ i f'(x)e^{i f(x)}$ and differentiating ${1 \over i f'(x)}$. 
We obtain
\[\int_c^d e^{i f(x)}\,dx = {e^{if(d)} \over i f'(d)} -  {e^{if(c)} \over i f'(c)}  + {1 \over i} \int_c^d  e^{i f(x)} { f''(x) \over  (f'(x))^2}\,dx \tag{1.2}\]
In view of the form of $(1.2)$, one might consider the integral on the right hand side of $(1.2)$ to be an improvement over the one on the left if 
$\big|{f''(x) \over (f'(x))^2}\big| < 1$ on $(c,d)$. This suggests the possibility of dividing the overall integral into two terms. The first term is the
 integral over those $x$ for which  $\big| { f''(x) \over  (f'(x))^2}\big| \geq 1$ (or where $f'(x) = 0$), a domain on which
no integration by parts is needed. The domain of the  second term  is where $\big| { f''(x) \over  (f'(x))^2}\big|  < 1$. This
 domain is a union of intervals on each of which we wish to integrate by parts as in $(1.2)$.

We will see that as long as some $k$th derivative $f^{(k)}$ is continuous and nonzero on $[a,b]$, the endpoint terms in the integrations by parts of 
$(1.2)$ are,
generally speaking, of no greater order of magnitude than the integral term. There is one exception, a term of magnitude $O( {1 \over \sup_{[a,b]}|f'(x)|})$
that may be incurred in the overall sum. The precise theorem is as follows.
\begin{theorem} 
Suppose $k$ is a positive integer and $f$ is a $\max(k,2)$ times continuously differentiable real-valued function on an interval $[a,b]$ such that there is a constant $A$ with
$\sup_{[a,b]} |f^{(k)}(x)| < A \inf_{[a,b]} |f^{(k)}(x)|$ $($note that this implies that $f^{(k)}(x)$ is nonzero on $[a,b])$. Suppose $\phi(x)$ is a $C^1$ function on $[a,b]$. Then we have the following.

\begin{itemize}
\item If $k \geq 2$ then there is a constant $B_{A,k} > 0$ such that 
\[\bigg|\int_a^b e^{i f(x)}\phi(x)\,dx\bigg| \leq \]
\[ B_{A,k} (||\phi||_{L^{\infty}} + ||\phi'||_{L^1})\bigg(\int_a^b\min\bigg(1, \bigg|{f''(x) \over (f'(x))^2}\bigg|\bigg) dx + \min\big( b - a, {1 \over \sup_{[a,b]}|f'(x)|}\big) \bigg)
 \tag{1.3}\]
\item If $k = 1$ and $[a,b]$ is the union of $j$ intervals on which $f'(x)$ is monotonic, then $(1.3)$ holds with $B_{A,1}$ replaced by $jB_{A,1}$.
\end{itemize}
\end{theorem}
\noindent We will often use $(1.3)$ with $f(x)$ replaced by $\lambda f(x)$ for a real parameter $\lambda$. In this case $(1.3)$ becomes
\[\bigg|\int_a^b e^{i\lambda  f(x)}\phi(x)\,dx\bigg| \leq \]
\[B_{A,k} (||\phi||_{L^{\infty}} + ||\phi'||_{L^1}) \bigg(\int_a^b\min\bigg(1, \bigg|{f''(x) \over \lambda (f'(x))^2}\bigg|\bigg) dx + \min\big( b - a, {1 \over |\lambda| \sup_{[a,b]}|f'(x)|}\big) \bigg)
 \tag{1.4}\]

The condition that $\sup_{[a,b]} |f^{(k)}(x)| < A \inf_{[a,b]} |f^{(k)}(x)|$ means that $f(x)$ is of "polynomial type" in the sense of [PhSt], which enables
us to use properties of such functions (see Lemma 1 of Part 2 of [PhSt]) that are key in the proof of Theorem 1.1.

One might ask what happens if we "recreate" endpoint terms by integrating ${f''(x) \over \lambda (f'(x))^2}$ to 
$-{1 \over \lambda f'(x)}$ in $(1.4)$.
One sees that one gets terms of magnitude ${1 \over |\lambda f'(x)|}$ for $x = a$, $x = b$, any $x$ satisfying $f''(x) = 0$, or any $x$ satisfying 
$|\lambda| (f'(x))^2 = |f''(x)|$. This last condition can be restated as ${1 \over |\lambda f'(x)|} = {1 \over |\lambda|^{1 \over 2} |f''(x)|^{1 \over 2}}$.
Such endpoint terms are not easily seen to behave well under perturbations of $f(x)$, but might be of interest in individual integrals.

\subsection {A variant of Theorem 1.1 with weaker hypotheses.}

If one does not have the condition that $\sup_{[a,b]} |f^{(k)}(x)| < A \inf_{[a,b]} |f^{(k)}(x)|$ for some $A$ and some $k$ in Theorem 1.1, sometimes
one can simply divide $[a,b]$ into several intervals for which such a condition holds. But even if this is not the case, such as when $f(x)$ has a zero of
infinite order, one can still
obtain a statement resembling that of Theorem 1.1. But there are endpoint terms cropping up that may not behave well
under perturbations of the phase. The theorem is as follows.

\begin{theorem} Suppose $f$ is a $C^3$ function $[a,b]$ such that $\{x \in [a,b]: f'''(x) = 0\}$ is finite. Suppose $\phi(x)$ is a $C^1$ function on $[a,b]$.
Let $J$ denote the $($finite$)$ set of points $x$ in $[a,b]$ for which $f'(x) \neq 0$ and for which either $x = a$, $x = b$, $f''(x) = 0$, or $f'''(x) = 0$. 
Then there is a uniform constant $B$ such that 
\[\bigg|\int_a^b e^{i f(x)}\phi(x)\,dx\bigg| \leq \]
\[ B (||\phi||_{L^{\infty}} + ||\phi'||_{L^1})\bigg(\int_a^b\min\bigg(1, \bigg|{f''(x) \over (f'(x))^2}\bigg|\bigg) dx + \min\big( b - a, 
\sum_{x \in J} {1 \over |f'(x)|} \big) \bigg) \tag{1.5}\]
\end{theorem}
Although the conclusion of Theorem 1.2 is weaker than that of Theorem 1.1, it can still provide relevant information. Suppose we replace $f(x)$ by 
$\lambda f(x)$ in $(1.5)$, where $\lambda$ is a real parameter. Then $(1.5)$ becomes
\[\bigg|\int_a^b e^{i \lambda f(x)}\phi(x)\,dx\bigg| \leq \]
\[ B (||\phi||_{L^{\infty}} + ||\phi'||_{L^1})\bigg(\int_a^b\min\bigg(1, \bigg|{f''(x) \over \lambda (f'(x))^2}\bigg|\bigg) dx + \min\big( b - a, 
\sum_{x \in J} {1 \over |\lambda f'(x)|} \big) \bigg) \tag{1.6}\]
Note that if the decay rate of the integral term in $(1.6)$ is slower than $|\lambda|^{-1}$ as $|\lambda| \rightarrow \infty$, as is often the case, 
then the additional terms on the right will not affect the overall decay rate.

\noindent {\bf Example 1.}  On $(0, {1 \over 2}]$, let $f(x) = e^{-(|\ln x|^k)}$ for some $k > 1$. If we define $f(0) = 0$, then $f(x)$ is smooth on $[0,{1 \over 2}]$ with a zero of infinite
order at $x = 0$. One can directly compute 
\[f'(x) = -k e^{-|\ln x|^k}  \bigg({|\ln x|^{k - 1} \over x}\bigg) \tag{1.7a}\]
\[f''(x) = k^2 e^{-|\ln x|^k}  \bigg({|\ln x|^{2k - 2} \over x^2} + o\bigg({|\ln x|^{2k - 2} \over x^2}\bigg)\bigg)\tag{1.7b}\]
\[f''''(x) = -k^3 e^{-|\ln x|^k} \bigg({|\ln x|^{3k - 3} \over x^3} + o\bigg({|\ln x|^{3k - 3} \over x^3}\bigg)\bigg) \tag{1.7c}\]
By $(1.7b)-(1.7c)$,  if $c > 0$ is sufficiently small then Theorem 1.2 applies with only an endpoint term at $x = c$, where the error term is $O(|\lambda|^{-1})$, much smaller
than the contribution we will obtain from the integral. By $(1.7a)-(1.7b)$, we have
\[\frac{f''(x)}{(f'(x))^2} = \frac{1}{ e^{-|\ln x|^k}}(1 + o(1)) \]
Consequently, on $[0,c]$ for $c$ sufficiently small,  by Theorem 1.2 we have
\[\bigg|\int_0^c e^{i\lambda  e^{-|\ln x|^k} }\phi(x)\,dx\bigg| \leq \]
\[ 2B (||\phi||_{L^{\infty}} + ||\phi'||_{L^1})\bigg(\int_0^c \min\bigg(1,\frac{1}{ |\lambda| e^{-|\ln x|^k}}\bigg)\, dx  + O(|\lambda|^{-1})\bigg) \tag{1.8}\]
We split the integral in $(1.8)$ dyadically as 
\[ \int_0^c \min\bigg(1,\frac{1}{ |\lambda| e^{-|\ln x|^k}}\bigg)\, dx = \]
\[|\{x \in [0,c]: e^{-|\ln x|^k} < |\lambda|^{-1}\}\big| + 
\sum_{i = 0}^{\infty} 2^{-i}|\{x \in [0,c]: 2^i|\lambda|^{-1} \leq e^{-|\ln x|^k} < 2^{i+1}|\lambda|^{-1}\}| \tag{1.9}\]
Since the sublevel measure sets of $e^{-|\ln x|^k}$ grow more slowly than that of $x^n$ for any $n$, the sum in $(1.9)$ decreases rapidly in $i$ and the overall sum is bounded by a 
constant times that of the first term, which in turn is comparable to the $|\{x \in [0,c]: e^{-|\ln x|^k} < |\lambda|^{-1}\}\big| $ term.  Since the endpoint term is much smaller, the oscillatory 
integral $(1.6)$ satisfies the sublevel set bound
\[\bigg|\int_0^c e^{i\lambda  e^{-|\ln x|^k} }\phi(x)\,dx\bigg| \leq  B' (||\phi||_{L^{\infty}} + ||\phi'||_{L^1}) \big|\{x \in [0,c]: e^{-|\ln x|^k} < |\lambda|^{-1}\}\big| \tag {1.10}\]
This can be considered a canonical sublevel set bound since the expression $|\{x \in [0,c]: e^{-|\ln x|^k} < |\lambda|^{-1}\}|$ up to a constant gives the maximum size of a period of
the phase function $e^{i\lambda  e^{-|\ln x|^k}}$.

Working out the measure on the  right-hand side of $(1.10)$, we see that it equals the measure of the set where
$-|\ln x|^k < - \ln |\lambda|$, which is the same as the measure of the set where $-\ln x = |\ln x| > (\ln |\lambda|)^{1 \over k}$, namely $e^{-(\ln |\lambda|)^{1 \over k}}$. 

\noindent {\bf Example 2.}  On $(0, {1 \over 2}]$, let $f(x) =e^{-(x^{-m})}$ for some $m > 0$. Like in the last example, $f(x)$ extends to a smooth function on
$[0,{1 \over 2}]$ with a zero of infinite order at $x = 0$. This time we compute
\[ f'(x) = me^{-(x^{-m})}x^{-m - 1} \tag{1.11a}\]
\[f''(x) = -m^2 e^{-(x^{-m})}(x^{-2m - 2} + o(x^{-2m -2}))\tag{1.11b}\]
\[f'''(x) = m^3 e^{-(x^{-m})}(x^{-3m - 3} + o(x^{-3m - 3})) \tag{1.11c)}\]
Thus we have
\[\frac{f''(x)}{(f'(x))^2} =  \frac{1}{e^{-(x^{-m})}}(1 + o(1))\]
Hence similarly to the last example, on $[0,c]$ for $c$ sufficiently small there is just the one negligible endpoint term in Theorem 1.2 and we have
\[\bigg|\int_0^c e^{i\lambda e^{-(x^{-m})} }\phi(x)\,dx\bigg| \leq \]
\[ 2 B (||\phi||_{L^{\infty}} + ||\phi'||_{L^1})\int_0^c \min\bigg(1,\frac{1}{ |\lambda| e^{-(x^{-m})}}\bigg)\, dx  \tag{1.12}\]
Since the sublevel set measures of $e^{-(x^{-m})}$ once again grow slower than that of any $x^n$, exactly as in the steps leading to $(1.10)$ we again have the canonical 
sublevel set bound 
\[\bigg|\int_0^c e^{i\lambda  e^{-(x^{-m})}}\phi(x)\,dx\bigg| \leq  B' (||\phi||_{L^{\infty}} + ||\phi'||_{L^1}) \big|\{x \in [0,c]: e^{-(x^{-m})} < |\lambda|^{-1}\}\big|
 \tag{1.13}\]
This time the measure in question works out to $(\ln |\lambda|)^{-{1 \over m}}$, which is therefore an upper bound for the oscillatory integral. 

\section{A result concerning the behavior of real analytic functions on parallel lines and some consequences.}

The following theorem follows fairly quickly from Lemma 6.17 of [Mi]. But for completeness we include a proof in section 6 which shows how
it can be derived in relatively short order from resolution of singularities (and the proof will not be all that different from the proof of Lemma 6.17 in [Mi]).

\begin{theorem}
Suppose $g(x_1,...,x_n)$ is a real analytic function defined on a neighborhood of the origin, not identically zero. Then there is an $n-1$-dimensional ball
 $B_{n-1}(0,\eta)$ and a $k \geq 0$ such that for each $(x_1,...,x_{n-1})$ in $B_{n-1}(0,\eta)$ either $g(x_1,...,x_n) = 0$ for all $|x_n| < \eta$ or
  there is a $0 \leq l \leq k$, which may  depend on $(x_1,...,x_{n-1})$, such that  for all $|x_n| < \eta $ one has
\[0 <  {1 \over 2} |\partial_{x_n}^l g(x_1,...,x_{n-1},0)| <  |\partial_{x_n}^l g(x_1,...,x_n)| <  2|\partial_{x_n}^l g(x_1,...,x_{n-1},0)|\tag{2.1} \]
 The set of $(x_1,...x_{n-1})$ for which $g(x_1,...,x_n) = 0$ for all $|x_n| < \eta$ has measure zero.
\end{theorem}

It is important to note that in Theorem 2.1, $l$ can be (and often is) equal to zero.
We next give some corollaries of interest of Theorem 2.1. As a rather immediate consequence of Theorem 2.1 we have the following, which we point out is 
nontrivial only when $ \partial_{x_n}^k f(0,...,0,0)$ is zero for each $k \geq 1$.

\noindent {\bf Corollary 2.1.1.}  Suppose $f(x_1,...,x_n)$ is a real analytic function on a neighborhood of the origin, not identically zero.
 Then there is an $n-1$ dimensional ball $B_{n-1}(0,\eta)$
and a positive integer $p$ such that if $(x_1,...,x_{n-1}) \in B_{n-1}(0,\eta)$ and $s \in \R$, then either $f(x_1,...,x_n) = s$ for all $|x_n| < \eta$, or 
$f(x_1,...,x_n) = s$ has at most $p$ solutions $x_n$.

\begin{proof} If $\partial_{x_n} f(x_1,...,x_n)$ is identically zero, then the result is immediate, so we assume it is not identically zero. We apply Theorem 2.1 to
 $\partial_{x_n} f(x_1,...,x_n)$ and let $\eta$ be as in that theorem. If $(x_1,...,x_{n-1})$ is such that $\partial_{x_n} f(x_1,...,x_n)$ is identically zero in $x_n$, then
$f(x_1,...,x_n)$ is constant in $x_n$ and for each $s$ either $f(x_1,...,x_n) = s$ for all $|x_n| < \eta$ or $f(x_1,...,x_n) = s$ for no $x_n$. If 
on the other hand $(x_1,...,x_{n-1})$ is such that $\partial_{x_n} f(x_1,...,x_n)$ is not identically zero in $x_n$, then $(2.1)$ holds on $|x_n| < \delta$
for $g = \partial_{x_n}f$. Thus $(2.1)$ also holds on $|x_n| < \delta$ for $g = f$, where $l$ is replaced by $l + 1 \geq 1$. Thus $f$ has a nonvanishing 
derivative of order between $1$ and $k + 1$ in the $x_n$ variable. This  implies $f(x_1,...,x_n) = s $ has at most $k + 1$ solutions in the $x_n$ variable 
for this value of $(x_1,...,x_{n-1})$. This completes the proof of the corollary.
\end{proof}

\noindent As a consequence of Corollary 2.1.1, we have the following.

\noindent {\bf Corollary 2.1.2.} Suppose $f_1(x_1,...,x_n)$,...,$f_l(x_1,...,x_n)$ are real analytic functions on a neighborhood of the origin, none identically zero. Then there
is an $n-1$ dimensional ball $B_{n-1}(0,\eta)$ and a positive integer $p$ such that for each $s_1,...,s_l$ and each $(x_1,...,x_{n-1}) \in B_{n-1}(0,\eta)$,
 the set $\{x_n: |x_n| < \eta $ and $f_i(x_1,...,x_n) < s_i$ for each $i\}$ consists of at most $p$ intervals.

\begin{proof} By simply intersecting the sets in question, it suffices to take $l = 1$. In this case, for a given $(x_1,...,x_{n-1})$  the intervals comprising $\{x_n: |x_n| < \eta $ and $f_1(x_1,...,x_n) < s_1\}$ have a uniformly bounded number of endpoints by Corollary 2.1.1 and the result follows. 
\end{proof}

The next consequence of Theorem 2.1 is of the same general type as Theorem 1.2 of [ClMi] on decay rates of parameterized families of
oscillatory integrals. Similar to Corollary 2.1.1, the statement follows relatively quickly from the
Van der Corput lemma unless  $ \partial_{x_n}^k f(0,...,0,0)$ is zero for each $k \geq 1$.

\noindent {\bf Corollary 2.1.3.} Suppose $f(x_1,...,x_n)$ is a real analytic function on a neighborhood of some closed ball ${\bar B}(0,r)$ centered at the origin. Suppose
$\phi(x_1,...,x_n)$ is a $C^1$ function supported in $B(0,r)$. Define $I(\lambda)$ by
\[I(\lambda) = \int_{\R} e^{i \lambda f(x_1,...,x_n)} \phi(x_1,...,x_n)\,dx_n\tag{2.2}\]
Then there is a positive integer $p$ and a
function $A(x_1,...,x_{n-1})$, depending on $f$ and $r$, such that for all $(x_1,...,x_{n-1})$ either $f(x_1,...,x_n)$ is constant in $x_n$, or we have an inequality
\[|I(\lambda)| \leq A(x_1,...,x_{n-1}) \big(\sup_{x_n} |\phi(x_1,...,x_n)| + \int_{\R} |\partial_{x_n}\phi(x_1,...,x_n)|\,dx_n\big) |\lambda|^{-{1 \over p}} \tag{2.3}\]

\begin{proof} If $\partial_{x_n}^2 f(x)$ is identically zero, it follows immediately from applying the first derivative Van der Corput lemma in the form $(4.3')$
in the $x_n$ direction. If $\partial_{x_n}^2 f(x)$ is not identically zero, we apply Theorem 2.1 to $\partial_{x_n}^2 f(x)$ and use the resulting 
statement in conjunction 
with the Van der Corput lemma, again in the form $(4.3')$. On vertical lines where $\partial_{x_n}^2 f(x) = 0$ we again use the $p = 1$ case of the Van der 
Corput lemma, while on vertical lines where $(2.1)$ holds we use the Van der Corput lemma for $p \geq 2$. 
\end{proof}

\section{Averaging Theorem 1.1 over different directions.}

\subsection{Using polar coordinates to reduce to one-dimensional integrals.}

We return to the setting of $(1.1)$. Namely we look at the integral
\[I(\lambda) = \int  e^{i\lambda f(x_1,...,x_n)} \phi(x_1,...,x_n)\,dx_1\,...\,dx_n\tag{3.1}\]
Here $f(x_1,...,x_n)$ is a  nonconstant real-valued $C^2$ function on a neighborhood of a ball $|x| \leq b$  with $\nabla f(0) = 0$ 
and  $\phi(x_1,...,x_n)$ is a  $C^1$ function supported in  $|x| < b$.  We integrate $(3.1)$ in polar coordinates, obtaining
\[I(\lambda) = c_n \int_{S^{n-1}}\int_0^{\infty} e^{i\lambda f(r\omega)} r^{n-1}\phi(r\omega)\,dr \,d\omega\tag{3.2}\]
To deal with the $r^{n-1}$ factor in $(3.2)$, we will decompose the $r$ integral in $(3.2)$ dyadically in $r$. Namely we let $\psi(r)$ be a nonnegative smooth  compactly supported
 function on $[1/2, 2]$ such that $\sum_{m = -\infty}^{\infty} \psi(2^m r) = 1$ on $(0,\infty)$, and we let $\Phi_{m,\omega}(r) = \phi(r\omega)(2^m r)^{n-1} \psi(2^m r)$. 
Then  we have $r^{n-1}\phi(r\omega) = \sum_{m = - \infty}^{\infty} 2^{-m (n-1)} \Phi_{m,\omega}(r)$ and $(3.2)$ becomes 
\[I(\lambda) = c_n\sum_{m = -\infty}^{\infty}  \int_{S^{n-1}}2^{-m (n-1)}\int_{2^{-m-1}}^{2^{-m + 1}} e^{i\lambda f(r\omega)}  \Phi_{m,\omega}(r)\,dr \,d\omega\tag{3.3}\]
Note that $||\Phi_{m,\omega}(r)||_{L^{\infty}(r)}$ and $||\partial_r\Phi_{m,\omega}(r)||_{L^1(r)}$ are bounded uniformly in $\omega$ and $m$ by $C||\phi||_{L^\infty}$ and 
 $C(||\phi||_{L^\infty} + ||\nabla \phi||_{L^{\infty}})$ respectively for a uniform constant $C$. Thus we may apply Theorem 1.1 or 1.2 to the 
$r$ integral, and the factor $||\phi||_{L^{\infty}} + ||\phi'||_{L^1}$ appearing in these theorems can be replaced by $||\phi||_{L^\infty} + ||\nabla \phi||_{L^{\infty}}$.

\subsection{The real analytic situation.}

 When $f$ is real analytic and we are applying Theorem 1.1 in the above fashion,
we would also like the factor $B_{A,k}$ in Theorem 1.1 to be uniform in $k$ and $m$. This holds for the following reason. If one applies Theorem 2.1 to $h(x_1,...,x_n,r) = 
\partial_{r}(f(rx_1,...,rx_n))$ on a neigborhood of some $(x_1,...,x_n,r)$ with $r \geq 0$ and $(x_1,...,x_n) \in S^{n-1}$, then one obtains uniform constants  $B_{A,k}$  locally. 
Compactness will then give that these constants are uniform throughout the domain of integration of $(3.3)$. 

At first glance, one might think that in the above we might incur some
 constants $j B_{A,1}$ from the $p = 1$ case of Theorem 1.1 that are unbounded in $j$, but this is not the case. For by Theorem 2.1  applied to $\partial_r^2 (f(rx_1,...,rx_n)) = \partial_r h$,
 locally there is an $N$ such that for
 each $(x_1,...,x_n) \in S^{n-1}$ either $h(x_1,...,x_n,r)$ is constant in $r$ (when $\partial_r^2 (f(rx_1,...,rx_n)) $ is identically zero in $r$) or is the union
 of at most $N$ intervals on which $ \partial_r h$ is 
strictly monotonic in $r$ (when $(2.1)$ holds for $\partial_{r}^2 h$.) In either case, $j$ is always at most $N$ here.

One technical point worth mentioning here is that for one to be able to apply Theorem 1.1 in the current setting, one needs that
the set of $\omega \in S^{n-1}$ for which $f(r\omega)$ is constant in $r$ must have measure zero.  But this is true for any nonconstant
real analytic function; any $\omega$ for which $f(r\omega)$ is constant in $r$ is such that this constant is $f(0)$, and nonconstant real analytic functions take any given value on a set of
measure zero. 

We have the following result for the real analytic situation.

\begin{theorem}
Suppose $f(x)$ is a nonconstant real analytic function on a  neighborhood of a ball $|x| \leq  b$, where $b \leq 1$, such that $\nabla f(0) = 0$. Suppose $\phi(x)$ is $C^1$ and supported
 on  $|x| < 2^{-m_0 + 1}$ for some integer
$m_0$ with $2^{-m_0 + 1} < b$.  For a given integer $m \geq m_0$ and an $x$ with $|x| \in [2^{-m-1}, 2^{-m + 1}]$, let $J_m(x)$ denote the line segment from $2^{-m-1}\frac{x}{|x|}$ to
 $2^{-m+1}\frac{x}{|x|}$. 

There is a constant $C$ depending on $f$ such that the following holds, where $f_r$ denotes the derivative of $f$  in the radial direction. 
\[|I(\lambda)|  \leq C (||\phi||_{L^{\infty}} + ||\nabla \phi||_{L^{\infty}}|) \int_{|x| < 2^{-m_0 + 1}} \min\bigg(1, \bigg| \frac{f_{rr}(x)}{\lambda(f_r(x))^2}\bigg|\bigg)\,dx\,\,+ \]
\[C (||\phi||_{L^{\infty}} + ||\nabla \phi||_{L^{\infty}}|) \sum_{m \geq m_0}  \int_{2^{-m-1} < |x| < 2^{-m+1}} \min\bigg(1, \frac{1}{|\lambda| | x|\sup_{y \in J_m(x)} |f_r(y)|}\bigg)\,dx \tag{3.4}\]
\end{theorem}
\begin{proof}
We insert $(1.4)$ in the right-hand integrals of $(3.3)$ and then go back from polar to rectangular 
coordinates in the integral term of $(1.4)$. We obtain
\[|I(\lambda)|  \leq C (||\phi||_{L^{\infty}} + ||\nabla \phi||_{L^{\infty}}|) \int_{|x| < b} \min\bigg(1, \bigg| \frac{f_{rr}(x)}{\lambda (f_r(x))^2}\bigg|\bigg)\,dx  \]
\[+\,\,C (||\phi||_{L^{\infty}} + ||\nabla \phi||_{L^{\infty}}|) \sum_{m \geq m_0}2^{-m(n-1)}\int_{\omega \in S^{n-1}}\min\bigg(2^{-m+1}, \frac{1}{|\lambda|\sup_{r \in [2^{-m-1}, 2^{-m + 1}]} |f_r(r\omega)|}\bigg)\, d\omega \tag{3.5}\]
Here $C$ is a constant which depends on the function $f$. Note that
\[\min \bigg(2^{-m+1}, \frac{1}{|\lambda|\sup_{r \in [2^{-m-1}, 2^{-m + 1}]} |f_r(r\omega)|}\bigg) \]
\[\leq  2 \int_{2^{-m-1}}^{2^{-m + 1}}\min\bigg(1, \frac{1}{|\lambda| r\sup_{r \in [2^{-m-1}, 2^{-m + 1}]} |f_r(r\omega)|}\bigg)\,dr\tag{3.6}\]
For $x$ such that $2^{-m-1} < |x| < 2^{-m+1}$, the integral $(3.6)$ can be rewritten as
\[\int_{J_m(x)} \min\bigg(1, \frac{1}{|\lambda|r\sup_{y \in J_m(x)} |f_r(y)|}\bigg)\,dr\tag{3.7}\]
We now insert $(3.7)$ into $(3.5)$ and go back from polar to rectangular coordinates. The result is $(3.4)$ and we are done.
\end{proof}

To use $(3.4)$  (or $(1.4)$ for that matter) it is helpful to have a way to bound integrals such as the ones appearing in Theorem 3.1 by sublevel set measures. 
 A reasonably general statement of the type needed is given by the following relatively easy to prove lemma.

\begin{lemma} Let $(E,\mu)$ be a finite measure space and suppose $g(x)$ is a measurable function on $E$ such that for some positive constants
 $C$ and $\delta$,
 for all $t  > 0$ one has $\mu(\{x \in E: |g(x)| < t\}) \leq Ct^{\delta}$.  Suppose $\epsilon > 0$. Then there is a constant $D_{\delta, \epsilon} > 0$
such that the following holds for all $\lambda \neq 0$.
\begin{itemize}
\item If $\delta  < \epsilon$, then $\int_E \min(1,  |\lambda g|^{-\epsilon})\,d\mu < 
C D_{\delta,\epsilon}|\lambda|^{-\delta}$
\item If $\delta = \epsilon$, then $\int_E \min(1,|\lambda g|^{-\epsilon})\,d\mu < 
C D_{\epsilon,\epsilon}(1 + \log_+ |\lambda|)|\lambda|^{-\epsilon} + \mu(E)|\lambda|^{-\epsilon}$
\item If $\delta > \epsilon$, then $\int_E \min(1,|\lambda g|^{-\epsilon})\,d\mu < 
C(|\lambda|^{-\delta} +  D_{\delta,\epsilon}|\lambda|^{-\epsilon})  + \mu(E)|\lambda|^{-\epsilon}$
\end{itemize}

\noindent As a result, we have that
\begin{itemize}
\item If $\delta  \neq \epsilon$,  there is a $F_{\delta,\epsilon,g,E} > 0$ such that
$\int_E \min(1,  |\lambda g|^{-\epsilon})\,d\mu < F_{\delta,\epsilon,g,E} |\lambda|^{-\min(\delta,\epsilon)}$
\item If $\delta = \epsilon$, there is a $F_{\epsilon, g, E} > 0$ such that
$\int_E \min(1,  |\lambda g|^{-\epsilon})\,d\mu < F_{\epsilon,g,E}(1 + \log_+|\lambda|)|\lambda|^{-\epsilon}$
\end{itemize}
\end{lemma}

\begin{proof}
 It is natural to split the integral $\int_E \min(1,|\lambda g|^{-\epsilon})\,d\mu$ into $|g| \geq {1 \over |\lambda|}$
and $|g| < {1 \over |\lambda|}$ parts, obtaining
\[\int_E \min(1,|\lambda g|^{-\epsilon})\,d\mu = \mu(\{x \in E: |g(x)| < {1 \over |\lambda|}\}) + 
{1 \over |\lambda|^{\epsilon}} \int_{|g| \geq {1 \over |\lambda|}}|g|^{-{\epsilon}}\,d\mu \tag{3.8}\]
We first do the $\delta \geq \epsilon$ cases. We further divide the integral in $(3.4)$ into $|g| \geq 1$ and $|g| < 1$ parts. Using the bound
$\min(1,|\lambda g|^{-\epsilon}) \leq |\lambda|^{-\epsilon}$ on the $|g| \geq 1$ portion, we obtain
\[\int_E \min(1,|\lambda g|^{-\epsilon})\,d\mu \leq \mu(\{x \in E: |g(x)| <  {1 \over |\lambda|}\}) + 
\mu(E)|\lambda|^{-\epsilon}+ {1 \over |\lambda|^{\epsilon}} \int_{1 \geq |g| \geq {1 \over |\lambda|}}|g|^{-{\epsilon}}\,d\mu \tag{3.9}\]
If $|\lambda| < 1$ we take the right-hand integral in $(3.9)$ to be zero. By our assumptions, the first term on the right of $(3.9)$ is bounded by 
$C|\lambda|^{-\delta}$. We write the last term in $(3.9)$ as 
a sum in $j$ for $0 \leq j \leq  \lfloor \log_2 |\lambda| \rfloor$ of the integrals over the domains where 
$2^j|\lambda|^{-1} \leq |g(x)| < 2^{j+1}|\lambda|^{-1} $.  We obtain 
\[\int_E \min(1,|\lambda g|^{-\epsilon})\,d\mu \leq C|\lambda|^{-{\delta}} +  
\mu(E)|\lambda|^{-\epsilon} + \sum_{j=0}^{ \lfloor \log_2 |\lambda| \rfloor} |\lambda|^{-\epsilon} 
\int_{2^j|\lambda|^{-1} \leq |g(x)| < 2^{j+1}|\lambda|^{-1}}|g(x)|^{-\epsilon}\,d\mu\tag{3.10}\]
Using the hypothesis $\mu({x \in E: |g(x)| < a}) \leq Ca^{\delta}$ in each term of the sum $(3.10)$ we get that 
\[\int_E \min(1,|\lambda g|^{-\epsilon})\,d\mu \leq C|\lambda|^{-{\delta}} + \mu(E) |\lambda|^{-\epsilon} + 
 C \sum_{j=0}^{ \lfloor \log_2 |\lambda| \rfloor}|\lambda|^{-\epsilon}(2^j|\lambda|^{-1})^{-\epsilon} (2^{j+1}|\lambda|^{-1})^{\delta} \tag{3.11}\]
\[= C|\lambda|^{-{\delta}} + \mu(E) |\lambda|^{-\epsilon} +  2^{\delta}C \sum_{j=0}^{ \lfloor \log_2 |\lambda| \rfloor} 2^{j(\delta - \epsilon)}|\lambda|^{-\delta} \tag{3.12}\]
If $\delta = \epsilon$, the sum is bounded by $( \lfloor \log_2 |\lambda| \rfloor + 1)|\lambda|^{-\delta}$. Since
 $ \lfloor \log_2 |\lambda| \rfloor \leq  \log_2 e \log |\lambda|$, we get the second statement in Lemma 3.2. 
If $\delta > \epsilon$, the sum is bounded by a constant depending on $\delta$ and $\epsilon$ times $2^{ \lfloor \log_2 |\lambda| \rfloor(\delta - \epsilon)}
|\lambda|^{-\delta} = |\lambda|^{-\epsilon}$.
Inserting this into the sum of $(3.12)$ gives the third statement in Lemma 3.2.

Now we suppose $\delta < \epsilon$. This time we do not break into $|g| \geq 1$ and $|g| < 1$ parts and instead just do the dyadic decomposition in $|g(x)|$. Equation $(3.12)$ gets
replaced by
\[\int_E \min(1,|\lambda g|^{-\epsilon})\,d\mu \leq C|\lambda|^{-{\delta}} + 2^{\delta}C \sum_{j=0}^{\infty} 2^{j(\delta - \epsilon)}|\lambda|^{-\delta} \tag{3.13}\]
The sum in $(3.13)$ is bounded by a constant depending on $\delta$ and $\epsilon$ times $|\lambda|^{-\delta}$, giving 
the first statement in Lemma 3.2 and we are done.
\end{proof}

\noindent {\bf Example.} Let $f(x_1,...,x_n) = \sum_{\alpha} c_{\alpha}x^{\alpha}$ be a polynomial such that $c_{\alpha} > 0$ for all $\alpha$ and 
$\alpha = (\alpha_1,...,\alpha_n)$ where each 
$\alpha_{j}$ is even, with  no $\alpha$ being the zero vector. Observe that  $\partial_r x^{\alpha} = |x|^{-1} \sum_j x_j \partial_{x_j} x^{\alpha} =k_{\alpha}|x|^{-1}x^{\alpha}$ for $k_{\alpha}= \sum_j \alpha_{j} \geq 2$, so that $\partial_{rr} x^{\alpha} = -|x|^{-2}k_{\alpha}x^{\alpha} + |x|^{-2}k_{\alpha}^2x^{\alpha}
= l_{\alpha}|x|^{-2}x^{\alpha}$, where  $l_{\alpha} =k_{\alpha}^2 -k_{\alpha} > 0$. 
Inserting these into $(3.4)$, it is not hard to verify that both terms of $(3.4)$ are bounded by
\[C' (||\phi||_{L^{\infty}} + ||\nabla \phi||_{L^{\infty}}) \int_{|x| < b} \min\bigg(1,  \frac{1}{|\lambda f(x)|}\bigg)\,dx\tag{3.14}\]
If $f(x)$ satisfies sublevel set estimates of the form $|\{|x| < b: |f(x)| < t\}| < At^{\delta}$ for some $0 <  \delta < 1$, then Lemma 3.2 gives that $(3.14)$ is bounded by
$A'|\lambda|^{-{\delta}}$. Thus in this situation, we have the canonical bounding of an oscillatory integral by a sublevel set measure given by $|I(\lambda)| < 
A''(||\phi||_{L^{\infty}} + ||\nabla \phi||_{L^{\infty}}) |\lambda|^{-{\delta}}$. If the supremum of the $\delta$ for which the sublevel set measure bounds holds is less than $1$ then one 
cannot do better;  by resolution of singularities the best exponent
for the oscillatory integral will always be the same as the best exponent for the sublevel set measures. We refer to the textbook [AGuV] for more information on these issues.

\subsection{A variation on a theme and some infinitely flat examples.}

In $(3.2)$ we change variables in the $r$ integral, letting $s = r^n$. Correspondingly we define $g_{\omega}(s) = f(s^{\frac{1}{n}}\omega)$ and $\psi_{\omega}(s) = \phi(s^{\frac{1}{n}}\omega)$. Then $(3.2)$ becomes
\[I(\lambda) = \frac{c_n}{n} \int_{S^{n-1}}\int_0^{\infty} e^{i\lambda g_{\omega}(s)}\psi_{\omega}(s)\,ds \,d\omega\tag{3.15}\]
In the form $(3.15)$, we may apply Theorem 1.1 or 1.2 directly to the $s$ integral without any dyadic decomposition, and then integrate the result over $S^{n-1}$. We will see that
for certain classes of functions with a zero of infinite order at the origin, $(3.15)$ bounds $|I(\lambda)|$ by the appropriate sublevel set measure, analogous to the above example and the
examples at the end of section 1.

\noindent {\bf Example 1.} Let $f(x) = e^{-(|\ln |x||^{k})}$, where $k > 1$.
Then $g_{\omega}(s) = e^{-(\frac{|\ln s|}{n})^{k}}$. This time we have the formulas
\[{\partial g_{\omega} \over \partial s } = -\frac{k}{ns} e^{-(\frac{|\ln s|}{n})^{k}} \bigg(\frac{|\ln s|}{n}\bigg)^{k - 1} \tag{3.16a}\]
\[{\partial^2 g_{\omega} \over \partial s^2} =\frac{k^2}{n^2s^2} e^{-(\frac{|\ln s|}{n})^{k}} \bigg(\frac{|\ln s|}{n}\bigg)^{2k - 2} + o\bigg(\frac{k^2}{n^2s^2} e^{-(\frac{|\ln s|}{n})^{k}} \bigg(\frac{|\ln s|}{n}\bigg)^{2k - 2}\bigg) \tag{3.16b}\]
\[{\partial^3 g_{\omega} \over \partial s^3}= -\frac{k^3}{n^3s^3} e^{-(\frac{|\ln s|}{n})^{k}} \bigg(\frac{|\ln s|}{n}\bigg)^{3k - 3} + o\bigg(\frac{k^3}{n^3s^3} e^{-(\frac{|\ln s|}{n})^{k}} \bigg(\frac{|\ln s|}{n}\bigg)^{3k - 3}\bigg) \tag{3.16c}\]
By $(3.16a)-(3.16c)$, if $c$ is sufficiently small then ${\displaystyle {\partial g_{\omega} \over \partial s }, {\partial^2 g_{\omega} \over \partial s^2}}$, and ${\displaystyle {\partial^3 g_{\omega} \over \partial s^3 }}$ are nonzero on $[0,c]$ except at
$s = 0$. Therefore if $\phi$ is supported in $|x| < c$,  in each $s$ integral of $(3.15)$ we can apply Theorem 1.2 and the only endpoint term appearing will come from the right-hand endpoint, which will be $O(|\lambda|^{-1})$.
After integrating the result in $\omega$ we obtain
\[|I(\lambda)| \leq  C (||\phi||_{L^{\infty}} + ||\nabla \phi ||_{L^{\infty} })\bigg(\int_0^c \min\bigg(1,\frac{1}{ |\lambda| e^{-(\frac{|\ln s|}{n})^{k}}}\bigg)\, ds  + O(|\lambda|^{-1})\bigg)
 \tag{3.17}\]
We have a $||\nabla \phi ||_{L^{\infty} }$ factor here due to the presence of the $\psi_{\omega}(s)$ factor. 
Similar to the case of  Example 1 at the end of section 1, the $O(|\lambda|^{-1})$ term in $(3.17)$ is negligible compared to the 
integral term, and like before due to the slow growth rate of the measures of the  sublevel sets of $e^{-(\frac{|\ln s|}{n})^{k}}$, the integral is bounded by a constant times the
 measure of the set
where one takes the left component in the the minimum. Putting this all together we obtain the following, where $m$ denotes one-dimensional Lebesgue measure.
\[|I(\lambda)| \leq  C' (||\phi||_{L^{\infty}} + ||\nabla \phi ||_{L^{\infty} })m(s: e^{-(\frac{|\ln s|}{n})^{k}} < |\lambda|^{-1}) \tag{3.18}\]
We unwrap the measure on the right side of $(3.18)$. It is the same as the measure of the $s$ for which $-(\frac{|\ln s|}{n})^{k} < -\ln|\lambda|$, which in turn is the same as the 
measure of the set where $\frac{|\ln s|}{n} > (\ln|\lambda|)^{\frac{1}{k}}$. Since $|\ln s| = -\ln s$ here, this is the same as the measure of the set where 
$-\ln s > n [(\ln|\lambda|)^{\frac{1}{k}}]$, a set of measure $e^{-n [(\ln|\lambda|)^{\frac{1}{k}}]}$.

We next examine the $n$-dimensional measure of the set where $f(x) = e^{-|\ln |x||^{k}}$ is less than $|\lambda|^{-1}$. We have that $ e^{-|\ln |x||^{k}} < |\lambda|^{-1}$ exactly
when $-|\ln |x||^{k} < -\ln |\lambda|$, which happens when $|\ln |x|| > (\ln |\lambda|)^{\frac{1}{k}}$. Since $|\ln |x|| = -\ln |x|$, this occurs when $\ln |x| < -[(\ln |\lambda|)^{\frac{1}{k}}]$, or $|x| < e^{ -[(\ln |\lambda|)^{\frac{1}{k}}]}$, a set of measure a constant times $e^{ -n[(\ln |\lambda|)^{\frac{1}{k}}]}$. This is exactly the measure appearing in the last paragraph.
So $(3.18)$ implies that the oscillatory integral $I(\lambda)$ satsifies the following natural sublevel set bounds for sufficiently large $|\lambda|$, where $m_n$ denotes $n$-dimensional 
Lebesgue measure
\[|I(\lambda)| \leq C'' (||\phi||_{L^{\infty}} + ||\nabla \phi ||_{L^{\infty} })m_n(\{x:  e^{-|\ln |x||^{k}} < |\lambda|^{-1}\}) \tag{3.19}\]
As in the one-dimensional case, the oscillatory integral is seen to be bounded by a constant times the maximal measure of the set for which the exponential $e^{i\lambda e^{-|\ln |x||^{k}}}$ is in a 
single period. 

\noindent {\bf Example 2.} Let $f(x) = e^{-(|x|^{-m})}$ for some $m > 0$. We proceed as in the last example. This time $g_{\omega}(s) = e^{-({ s^{-\frac{m}{n}}})}$ and we
have the  formulas
\[{\partial g_{\omega} \over \partial s }  = -\frac{m}{n} \frac{e^{-({ s^{-\frac{m}{n}}})}}{s^{{m \over n} + 1}} \tag{3.20a}\]
\[{\partial^2 g_{\omega} \over \partial s^2} =  \frac{m^2}{n^2} \frac{e^{-({ s^{-\frac{m}{n}}})}}{s^{2{m \over n} + 2}} + o\bigg(\frac{m^2}{n^2} \frac{e^{-({ s^{-\frac{m}{n}}})}}{s^{2{m \over n} + 2}}\bigg) \tag{3.20b} \]
\[{\partial^3 g_{\omega} \over \partial s^3}= -\frac{m^3}{n^3} \frac{e^{-({ s^{-\frac{m}{n}}})}}{s^{3{m \over n} + 3}} + o\bigg(\frac{m^3}{n^3} \frac{e^{-({ s^{-\frac{m}{n}}})}}{s^{3{m \over n} + 3}}\bigg) \tag{3.20c}\]
Equations $(3.20a)-(3.20c)$ show that on a small enough neighborhood of the origin, the functions ${\displaystyle {\partial g_{\omega} \over \partial s },
 {\partial^2 g_{\omega} \over \partial s^2}}$, and ${\displaystyle {\partial^3 g_{\omega} \over \partial s^3 }}$  are
 nonzero except at $s = 0$. So like in the last example, in each $s$ integral of $(3.15)$ we may apply Theorem 1.2 and the only endpoint term appearing is the $O(|\lambda|^{-1})$ term
from the right endpoint. After integrating the result in $\omega$, in analogy with $(3.17)$ we have
\[|I(\lambda)| \leq  C (||\phi||_{L^{\infty}} + ||\nabla \phi ||_{L^{\infty} })\bigg(\int_0^c \min\bigg(1,\frac{1}{ |\lambda| e^{-({ s^{-\frac{m}{n}}})}}\bigg)\, ds  + O(|\lambda|^{-1})\bigg)
 \tag{3.21}\]
Again the $O(|\lambda|^{-1})$ term will be negligible compared to the 
integral term, and like before due to the slow growth rate of the sublevel sets of $e^{-({ s^{-\frac{m}{n}}})}$, the integral is bounded by a constant times the measure of the set
where one takes the left component in the the minimum. As a result, in place of $(3.18)$ we get
\[|I(\lambda)| \leq  C' (||\phi||_{L^{\infty}} + ||\nabla \phi ||_{L^{\infty} })m(s: e^{-({ s^{-\frac{m}{n}}})} < |\lambda|^{-1}) \tag{3.22}\]
Unwrapping the measure on the right of $(3.22)$, we have that it equals the measure of the set where $-(s^{-\frac{m}{n}}) < - \ln |\lambda|$, or equivalently the set where 
$s^{-\frac{m}{n}} > \ln |\lambda|$, which is the same as $s < (\ln |\lambda|)^{-\frac{n}{m}}$. So the measure on the right-hand side of $(3.22)$ is $(\ln |\lambda|)^{-\frac{n}{m}}$. On the
other hand, the $n$ dimensional measure of the set where $f(x) =  e^{-(|x|^{-m})}$ satisfies $|f(x)| < |\lambda|^{-1}$ is the same as the measure of the set where 
$-(|x|^{-m}) < -\ln |\lambda|$, which is the same as the measure of the set where $|x|^{-m} > \ln |\lambda|$, or equivalently where $|x| < (\ln |\lambda|)^{-\frac{1}{m}}$. Since this is in
$n$ dimensions, the measure here is a constant times $ (\ln |\lambda|)^{-\frac{n}{m}}$, same as above. So in analogy with $(3.19)$, for large enough $|\lambda|$ we have the following,
 where once again $m_n$ denotes $n$-dimensional Lebesgue measure.
\[|I(\lambda)| \leq C'' (||\phi||_{L^{\infty}} + ||\nabla \phi ||_{L^{\infty} })m_n(\{x:  e^{-|x|^{-m}} < |\lambda|^{-1}\}) \tag{3.23}\]
These are the natural bounds for an oscillatory integral in terms of the sublevel set measures of its phase function.

\section{Efficiency of Theorems 1.1 and 1.2.}

 In this section we describe a few ways in which Theorems 1.1 and 1.2 can be viewed as efficient. For one, Theorem 1.1 minimizes the presence of endpoint terms that appear in integrations by parts for such one-dimensional oscillatory integrals, having just the one term on the right of $(1.4)$. 

Next, suppose we are in the setting of Theorem 1.1 and
$f'(x_0) = 0$ for some $x_0 \in [a,b]$. Let $p$ denote the order of the zero of $f'$ at $x_0$; by the assumptions of Theorem 1.1 we must have $1 \leq p
\leq k-1$ where $k$ is as in the theorem. Thus near $x_0$ one has $f'(x) = \beta(x - x_0)^p + o((x- x_0)^p)$ for some nonzero constant $\beta$. We examine
the effect this zero of $f'$ has on the right-hand side of $(1.3)$. Note that $f''(x) = \beta p (x - x_0)^{p-1} +  o((x- x_0)^{p-1})$.
 Thus near $x_0$ we have
\[{f''(x) \over  (f'(x))^2} = {p  \over \beta } {1 \over (x - x_0)^{p+1}} + o\bigg( {1 \over (x - x_0)^{p+1}}\bigg) \tag{4.1}\]
Thus as $|\lambda| \rightarrow \infty$ the portion of the integral term in $(1.4)$ near $x_0$ behaves as the following expression for a small but
 fixed $\delta > 0$.
\[\int_{|x - x_0| < \delta} \min\bigg(1,   {p  \over |\beta|} {1 \over |\lambda| |x - x_0|^{p+1}} \bigg) dx \tag{4.2}\]
Note that ${\displaystyle {p \over \beta} {1 \over |\lambda| |x - x_0|^{p+1}} = 1}$ when ${\displaystyle |x - x_0| =
\bigg({p \over \beta}\bigg)^{{1 \over p+1}} |\lambda|^{-{1 \over p+1}}}$. Thus as $|\lambda| \rightarrow \infty$, the integral $(4.2)$ is of the same order
as $|\beta|^{-{1 \over p + 1}} |\lambda|^{-{1 \over p + 1}}$. So the critical point at $x_0$ contributes 
$C(||\phi||_{L^{\infty}} + ||\phi'||_{L^1}) |\beta|^{-{1 \over p + 1}} |\lambda|^{-{1 \over p + 1}}$ to the right-hand side of $(1.4)$. This is the 
correct order of decay for the contribution this critical point makes
 to the oscillatory integral of $(1.4)$ as $\lambda \rightarrow \infty$, as long as $\phi(x_0) \neq 0$. Thus the right-hand side of $(1.4)$ properly
takes into account a critical point, as $|\lambda| \rightarrow \infty$. Although we still have the factor $||\phi||_{L^{\infty}} + ||\phi'||_{L^1}$,
if we replace the interval $[a,b]$ by an appropriately small interval containing $x$, the factor $||\phi||_{L^{\infty}} + ||\phi'||_{L^1}$ will 
more closely approximate the correct coefficient $|\phi(x_0)|$.

 We should point out that the above discussion also holds for the variant of Theorem 1.1 given by
Theorem 1.2.

Another way in which Theorem 1.1 can be viewed as efficient is in the context of the Van der Corput lemma. One version of the
 standard form of the Van der Corput lemma can be stated as follows (see p.334 of Chapter 8 of [St]). 

\noindent  {\bf Van der Corput Lemma.} {\it Suppose $p$ is a positive integer and $f(x)$ is a $C^{\max(p,2)}$ function on an interval $[a,b]$ such that
$|f^{(p)}(x)| > 1$ on $[a,b]$. If $p = 1$ also assume $f'$ is monotonic. Suppose $\phi(x)$ is a 
$C^1$ function on $[a,b]$. Then there is a constant $c_p > 0$ such that }
\[\bigg|\int_a^b e^{i \lambda f(x)}\phi(x)\,dx\bigg| \leq c_p (||\phi||_{L^{\infty}} + ||\phi'||_{L^1}) |\lambda|^{-{1 \over p}} \tag{4.3} \]
Sometimes, the Van der Corput Lemma is stated under the assumption that $|f^{(p)}(x)| > A$ for some $A > 0$ instead of $|f^{(p)}(x)| > 1$. In this case
one may replace $f(x)$ by $A f(x)$ in the above, and the resulting conclusion is that
\[\bigg|\int_a^b e^{i \lambda f(x)}\phi(x)\,dx\bigg| \leq c_p (||\phi||_{L^{\infty}} + ||\phi'||_{L^1}) A^{-{1 \over p}}|\lambda|^{-{1 \over p}} \tag{4.3'} \]
There is also a sublevel set measure version of the Van der Corput lemma with a very similar proof which can be found in [C]. It can be stated as follows.

\noindent  {\bf Van der Corput lemma II.}  {\it Suppose $p$ is a positive integer and $f(x)$ is a $C^p$ function on an interval $[a,b]$ such that
$|f^{(p)}(x)| > B$ on $[a,b]$, where $B > 0$. Then there is a constant $c_p' > 0$ such that for each $\epsilon > 0$ we have }
\[|\{x \in I: |f(x)| < \epsilon\}| < c_p' {\bigg({\epsilon \over B}\bigg)}^{{1 \over p}} \tag{4.4}\]
Another notion of Theorem 1.1 being efficient is given by the following theorem pertaining to the right hand side of $(1.4)$.

\begin{theorem} Suppose that $f(x)$ is a polynomial and $p \geq  1$ is such that
$|f^{(p)}(x)| > 1$ on $[a,b]$. Let $l$ denote the degree of $f(x)$. Then there is a positive constant $C_p$ such that 
\[\int_a^b\min\bigg(1, \bigg|{f''(x) \over \lambda (f'(x))^2}\bigg|\bigg) dx + \min\bigg( b - a, {1 \over |\lambda| \sup_{[a,b]}|f'(x)|}\bigg) \leq l C_p 
|\lambda|^{-{1 \over p}} \tag{4.5}\]
\end{theorem} 
Note that in the polynomial case, the $k$ in $(1.4)$ can be taken to
be the degree $l$ of $f(x)$. Thus
Theorem 4.1 can be viewed as saying that for polynomials, the estimate given by $(1.4)$ contains the information given by the Van der Corput lemma,
 other than the fact that the constant appearing may depend on the degree of the polynomial as well as $p$. This dependence gibes with the 
condition in Theorem 1.1 that $\sup_{[a,b]} |f^{(k)}(x)| < A \inf_{[a,b]} |f^{(k)}(x)|$.

\noindent{\bf Proof of Theorem 4.1.}

We first focus on the term $ \min( b - a, {1 \over |\lambda| \sup_{[a,b]}|f'(x)|}) $ in $(4.5)$. If $p = 1$, this term is bounded by ${1 \over |\lambda|}$, 
which gives the estimate we need. So we assume $p > 1$. We apply $(4.4)$ to $f'(x)$, specifically that $\{x \in [a,b]: |f'(x)| < {1 \over (c_{p-1}')^{p-1}}
(b-a)^{p-1}\}$ has measure less than $b - a$. Thus $\sup_{[a,b]}|f'(x)| \geq {1 \over (c_{p-1}')^{p-1}} (b-a)^{p-1}$. Hence we have
\[\min( b - a, {1 \over |\lambda| \sup_{[a,b]}|f'(x)|}) \leq \min(b -a, {(c_{p-1}')^{p-1} \over |\lambda| (b-a)^{p-1}}) \tag{4.6}\]
Viewed as functions of $b - a$ for fixed $\lambda$, the left term in the minimum is increasing and the right term is decreasing.
The two terms are equal when $b - a = (c_{p-1}')^{p-1 \over p}|\lambda|^{-{1 \over p}}$. Thus we have the desired estimate
\[\min( b - a, {1 \over |\lambda| \sup_{[a,b]}|f'(x)|}) \leq (c_{p-1}')^{p-1 \over p}|\lambda|^{-{1 \over p}} \tag{4.7}\]
We now move on to the integral term $\int_a^b\min\big(1, \big|{f''(x) \over \lambda (f'(x))^2}\big|\big) dx $ in $(4.5)$. Again we first do the case
 $p = 1$ separately. When $p = 1$, we have 
\[\int_a^b\min\bigg(1, \bigg|{f''(x) \over \lambda (f'(x))^2}\bigg|\bigg) dx \leq {1 \over |\lambda|} \int_a^b \bigg|{f''(x) \over  (f'(x))^2}\bigg| dx \tag{4.8}\]
Since $f(x)$ is a polynomial, $f'(x)$ is piecewise monotonic, and the integral in $(4.8)$ can be integrated over each piece to obtain the sum of $\pm 
{1 \over f'(x)}$ at the two endpoints. Thus by the condition that $|f'| > 1$, the integral over each piece less than $2$, and thus the overall integral is 
bounded by $2l$, where we recall $l$ is the degree of $f(x)$. Thus for the case where $p = 1$ we have the desired estimate
\[\int_a^b\min\bigg(1, \bigg|{f''(x) \over \lambda (f'(x))^2}\bigg|\bigg) dx \leq 2l {1 \over |\lambda|}\tag{4.9}\]
Henceforth we assume $p > 1$ and focus on bounding  $\int_a^b\min(1, |{f''(x) \over \lambda (f'(x))^2}|) dx $ when $p > 1$. Note that the degree $l$ of
the polynomial $f(x)$ is at least 2. Thus we may factorize $f'(x) = \gamma\prod_{i=1}^{l-1}(x - r_i)$, where the $r_i$ are the roots of $f'(x)$ and
$\gamma \neq 0$. Note that
${f''(x) \over f'(x)}$ is the logarithmic derivative of $f'(x)$, or $\sum_{i=1}^{l-1} {1 \over x - r_i}$, so that we have 
\[{f''(x) \over (f'(x))^2} = \sum_{i=1}^{l-1} {1 \over (x-r_i)f'(x)} \tag{4.10}\]
Due to the elementary inequality $\min(a,b+c) \leq \min(a,b) + \min(a,c)$ for nonnegative $a,b$, and $c$, we then get
\[\int_a^b\min\bigg(1, \bigg|{f''(x) \over \lambda (f'(x))^2}\bigg|\bigg) dx \leq \sum_{i=1}^{l-1} \int_a^b\min\bigg(1, {1 \over |\lambda (x-r_i)f'(x)|}\bigg) \tag{4.11}\]
By $(4.4)$, for each $\epsilon > 0$ one has
\[|\{x \in [a,b]: |f'(x)| < \epsilon\}| < c_{p-1}' {\epsilon}^{{1 \over p-1}} \tag{4.12}\]
This implies corresponding bounds for each $|\{x \in [a,b]: |(x - r_i) f'(x)| < \epsilon\}|$. To see why, note that 
\[\{x \in [a,b]: |(x - r_i)f'(x)| < \epsilon\} \subset \{x \in [a,b]: |(x - Re(r_i)) f'(x)| < \epsilon\}\]
\[\subset \cup_{j= -\infty}^{\infty} \{x \in [a,b]: |x - Re(r_i)| < 2^j, |f'(x)| < {\epsilon \over 2^j}\} \tag{4.13}\]
Note that $|\{x \in [a,b]: |x - Re(r_i)| < 2^j \}| < 2^{j+1}$, while by $(4.12)$ we have 
\[|\{x \in [a,b]: |f'(x)| < {\epsilon \over 2^j}\}|  \leq  c_{p-1}' {\epsilon}^{{1 \over p-1}}2^{-{j \over p -1}} \tag{4.14}\]
Thus $(4.13)$ gives
\[|\{x \in [a,b]: |(x - r_i)f'(x)| < \epsilon\}| \leq \sum_{j= -\infty}^{\infty} \min(2^{j+1}, c_{p-1}' {\epsilon}^{{1 \over p-1}}2^{-{j \over p -1}})\tag{4.15}\]
One computes that the terms $2^{j+1}$ and $c_{p-1}' {\epsilon}^{{1 \over p-1}}2^{-{j \over p -1}}$ in the minimum of $(4.15)$ are equal when 
$2^j = ({c_{p-1}' \over 2})^{p-1 \over p} \epsilon^{1 \over p}$. Since $2^j$ increases exponentially in $j$ and  $c_{p-1}' {\epsilon}^{{1 \over p-1}}2^{-{j \over p -1}}$ decreases exponentially in $j$, we conclude that for some constant $d_p$ we have
\[|\{x \in [a,b]: |(x - r_i)f'(x)| < \epsilon\}| \leq d_p \epsilon^{1 \over p}\tag{4.16}\]
Equation $(4.16)$ holds for all $\epsilon > 0$. Thus by Lemma 3.2, there is a constant $e_p$ such that for each $i$  we have
\[ \int_a^b\min\bigg(1, {1 \over |\lambda (x-r_i)f'(x)|}\bigg)\,dx \leq e_p|\lambda|^{-{1 \over p}}\tag{4.17}\]
In view of $(4.11)$, adding $(4.17)$ over all $i$ gives
\[\int_a^b\min\bigg(1, \bigg|{f''(x) \over \lambda (f'(x))^2}\bigg|\bigg) dx \leq (l-1)e_p |\lambda|^{-{1 \over p}}\tag{4.18}\]
This gives us the desired bound for the term $\int_a^b\min(1, |{f''(x) \over \lambda (f'(x))^2}|) dx $ in $(4.5)$ and we are done. \qed

\section{Proofs of Theorems 1.1 and 1.2.}

\subsection{Proof of Theorem 1.1.} We will do the case when $k = 1$ separately at the end. So for now we assume the hypotheses of Theorem 1.1 hold 
for some $k \geq 2$.  Let $Z$ be the set
$\{x \in [a,b]: x = a, x = b,$ or $f'(x) = 0\}$. Since the assumptions of the theorem imply that $|f^{(k)}(x)| \neq 0$ on $[a,b]$, the set $Z$ is finite.
We next write $[a,b] 
= G \cup H$, where $G = \{x \in [a,b]: x  \in Z$ or $\big|{f''(x) \over (f'(x))^2}\big| \geq 1\}$, and $H = \{x \in [a,b]: x \notin Z$ and 
$\big|{f''(x) \over (f'(x))^2}\big| < 1\}$. Note that $H$ is open and therefore $G$ is closed.
Since the integrand of $(1.1)$ has magnitude at most $||\phi||_{L^{\infty}}$, the portion of the integral $(1.1)$ over $G$
is bounded by $||\phi||_{L^{\infty}} |G|$, which equals
\[||\phi||_{L^{\infty}} \int_G\min\bigg(1, \bigg|{f''(x) \over (f'(x))^2}\bigg|\bigg) dx\tag{5.1}\]
As a result we have
\[\bigg|\int_G e^{i f(x)}\phi(x)\,dx\bigg| \leq ||\phi||_{L^{\infty}} \int_G\min\bigg(1, \bigg|{f''(x) \over (f'(x))^2}\bigg|\bigg) dx\tag{5.2}\]
This is the only bound we will need for the portion of the oscillatory integral over $G$.

We move on to $H$.
Since $H$ is open we may write $H = \cup_j I_j$ where $I_j = (c_j,d_j)$ are disjoint open intervals. By definition of $H$, one has
$f'(x) \neq 0$ for each $x \in I_j$. We claim that $f'(c_j)$ and $f'(d_j)$ are nonzero, so we have that ${1 \over f'(x)}$ is a $C^1$ function on $[c_j,d_j]$ for which we may apply 
integrations by parts such as in $(1.2)$.  To see why $f'(c_j)$ and $f'(d_j)$ are nonzero, first note that for small $\epsilon_1, \epsilon_2 > 0$ we have
\[{1 \over f'(d_j - \epsilon_2)} - {1 \over f'(c_j + \epsilon_1)} = - \int_{c_j + \epsilon_1}^{d_j - \epsilon_2} {f''(x) \over (f'(x))^2}\,dx\tag{5.3}\]
Since $\big|{f''(x) \over (f'(x))^2}\big| < 1$ on $(c_j,d_j)$, we must therefore have
\[\bigg|{1 \over f'(d_j - \epsilon_2)} - {1 \over f'(c_j + \epsilon_1)}\bigg| <  (d_j - c_j) \tag{5.4}\]
If $f'(c_j)$ were zero, we could take limits as $\epsilon_1 \rightarrow 0$ in $(5.4)$ and get a contradiction. Similarly if $f'(d_j)$ were zero we could take 
limits as $\epsilon_2 \rightarrow 0$ in $(5.4)$ and get a contradiction. Thus we have that $f'(c_j)$ and $f'(d_j)$ are nonzero as needed.

Next we divide the intervals $I_j$ into two types. We say $I_j$ is of type 1 if $\sup_{I_j} |f'(x)| > 2\inf_{I_j}|f'(x)|$, and we say $I_j$ is of type 2 if 
$\sup_{I_j}|f'(x)| \leq 2\inf_{I_j}|f'(x)|$. We start with the analysis on intervals of type $1$. We integrate by parts in the integral $\int_{c_j}^{d_j}
 e^{i f(x)} \phi(x)\,dx$, writing $e^{i f(x)} = if'(x) e^{if(x)} \times {1 \over i f'(x)}$, then integrating $if'(x) e^{if(x)}$ to $e^{if(x)}$ and differentiating the
rest. The result is
\[\int_{c_j}^{d_j} e^{i f(x)}\phi(x)\,dx = {e^{if(d_j)} \over i f'(d_j)}\phi(d_j) -  {e^{if(c_j)} \over i f'(c_j)}\phi(c_j) \]
\[ + {1 \over i} \int_{c_j}^{d_j}  e^{i f(x)} { f''(x) \over  (f'(x))^2}\phi(x)\,dx  -  {1 \over i} \int_{c_j}^{d_j} e^{i f(x)}{1 \over f'(x)}\phi'(x)\,dx  \tag{5.5}\]
We now bound $(5.5)$ by the absolute values of each term of $(5.5)$, putting the absolute values inside the integrals and bounding the occurrences of
$\phi$ by $||\phi||_{L^{\infty}}$. We obtain
\[\bigg |\int_{c_j}^{d_j} e^{i f(x)}\phi(x)\,dx\bigg| \leq ||\phi||_{L^{\infty}}\bigg({1 \over |f'(d_j)|} + {1 \over |f'(c_j)|} + \int_{c_j}^{d_j} \bigg|{ f''(x) \over  (f'(x))^2}\bigg|\,dx\bigg) + \int_{c_j}^{d_j} {1 \over |f'(x)|}|\phi'(x)|\,dx \tag{5.6}\]
To analyze $(5.6)$ further we use the following lemma.
\begin{lemma}
If $I_j = (c_j,d_j)$ is any type 1 interval, then each $y \in [c_j,d_j]$ satisfies
\[{1 \over |f'(y)|} < 3 \int_{c_j}^{d_j} \bigg|{ f''(x) \over  (f'(x))^2}\bigg|\,dx \tag{5.7}\]
\end{lemma}
\begin{proof}
Let $x_0 \in [c_j,d_j]$ be such that $|f'(x_0)| = \inf_{[c_j,d_j]} |f'(x)|$. Recall by the earlier discussion that $f'(x_0) \neq 0$. Then by the fundamental theorem of
calculus, ${1 \over f'(y)} = {1 \over f'(x_0)} - \int_{x_0}^y {f''(x) \over (f(x))^2}\,dx$, so that
\[{1 \over |f'(y)|} \leq  {1 \over |f'(x_0)|} +  \int_{c_j}^{d_j} \bigg|{ f''(x) \over  (f'(x))^2}\bigg|\,dx \tag{5.8}\]
Now let $x_1 \in [c_j,d_j]$ be such $|f'(x_1)| = \sup_{[c_j,d_j]} |f'(x)|$. Since $I_j$ is type $1$, one has  $\displaystyle{{1 \over |f'(x_1)|} < {1 \over 2|f'(x_0)|}}$. As a result we have
\[ {1 \over |f'(x_0)|} < 2\, \bigg|{1 \over f'(x_0)}  - {1 \over f '(x_1)}\bigg| \tag{5.9}\]
Again using the fundamental theorem of calculus, $(5.9)$ leads to
\[ {1 \over |f'(x_0)|} < 2\, \int_{c_j}^{d_j} \bigg|{ f''(x) \over  (f'(x))^2}\bigg|\,dx \tag{5.10}\]
Inserting $(5.10)$ into $(5.8)$ gives $(5.7)$ and we are done.
\end{proof}

\noindent Next, we insert $(5.7)$ into the first two terms on the right of $(5.6)$, obtaining
\[\bigg |\int_{c_j}^{d_j} e^{i f(x)}\phi(x) \,dx\bigg| \leq 7\,||\phi||_{L^{\infty}} \int_{c_j}^{d_j} \bigg|{ f''(x) \over  (f'(x))^2}\bigg|\,dx 
 + \int_{c_j}^{d_j} {1 \over |f'(x)|}|\phi'(x)|\,dx \tag{5.11}\]
Inserting $(5.7)$ into the rightmost term of $(5.11)$ now gives
\[\bigg |\int_{c_j}^{d_j} e^{i f(x)}\phi(x) \,dx\bigg| \leq (7\,||\phi||_{L^{\infty}} + 3||\phi'||_{L^1})\int_{c_j}^{d_j} \bigg|{ f''(x) \over  (f'(x))^2}\bigg|
\,dx \tag{5.12}\] 
Since $\big|{ f''(x) \over  (f'(x))^2}\big| < 1$ on $(c_j,d_j)$, the $\big|{ f''(x) \over  (f'(x))^2}\big|$ appearing in $(5.12)$ is equal to
 $\min(1, \big|{ f''(x) \over  (f'(x))^2}\big|)$. As a result, $(5.12)$ can be rewritten as
\[\bigg |\int_{c_j}^{d_j} e^{i f(x)}\phi(x) \,dx\bigg| \leq  (7\,||\phi||_{L^{\infty}} + 3||\phi'||_{L^1})\int_{c_j}^{d_j}
 \min\bigg(1,\bigg|{ f''(x) \over  (f'(x))^2}\bigg|\bigg)\,dx \tag{5.12'}\] 
The estimate $(5.12')$ gives the bounds we need for the portion of the oscillatory integral over $(c_j,d_j)$ when it is an interval of type 1.

We now assume $(c_j,d_j)$ is an interval of type 2. We do the case where $(c_j,d_j)$ is the entire interval $(a,b)$ first. This is the only situation where we 
contribute to the right-hand term of $(1.3)$. We integrate by parts as in $(5.5)$, and once again have the estimate $(5.6)$. Instead of using $(5.7)$, 
this time we use
the fact that $(a,b)$ is an interval of type 2 and thus have a bound ${1 \over |f'(x)|} \leq 2{ 1\over \sup_{[a,b]} |f'|}$ holding on $[a,b]$. We insert this 
bound into $(5.6)$, both in the two endpoint terms and in the integral of the rightmost term. We get
\[\bigg |\int_{a}^{b} e^{i f(x)}\phi(x) \,dx\bigg| \leq( 4\,||\phi||_{L^{\infty}} + 2||\phi'||_{L^1}) { 1\over \sup_{[a,b]} |f'(x)|} + ||\phi||_{L^{\infty}} 
\int_{a}^{b} \bigg|{ f''(x) \over  (f'(x))^2}\bigg|\,dx \] 
\[\leq ( 4\,||\phi||_{L^{\infty}} + 2||\phi'||_{L^1})\bigg( { 1\over \sup_{[a,b]} |f'(x)|} + 
\int_{a}^{b} \bigg|{ f''(x) \over  (f'(x))^2}\bigg|\,dx \bigg) \tag{5.13} \]
Since $\min(1, \big|{ f''(x) \over  (f'(x))^2}\big|) =  \big|{ f''(x) \over  (f'(x))^2}\big|$ in the situation at hand, the right-hand side of $(5.13)$ is the same as
\[  ( 4\,||\phi||_{L^{\infty}} + 2||\phi'||_{L^1})\bigg( { 1\over \sup_{[a,b]} |f'(x)|} + 
\int_{a}^{b} \min\bigg(1, \bigg|{ f''(x) \over  (f'(x))^2}\bigg|\bigg)\,dx \bigg)\tag{5.14}\] 
The estimate $(5.14)$ gives the bounds we need for the portion of the oscillatory integral over $(c_j,d_j)$ if $(c_j,d_j) =  (a,b)$ is an interval of type 2.

Next we assume that $(c_j,d_j)$ is an interval of type 2, and $(c_j,d_j)$ is not the whole interval $(a,b)$. Thus there is at least one endpoint $e$ of $[c_j,d_j]$
which is neither $a$ nor $b$. Since ${|f''(x)| \over (f'(x))^2} < 1$ on $(c_j,d_j)$ and $f'(e) \neq 0$, we must have ${|f''(e)| \over (f'(e))^2} \leq 1$. We
cannot have a strict inequality ${|f''(e)| \over (f'(e))^2} < 1$; if this were true
then $x$ would have to be an interior point of $(c_j,d_j)$ instead of an endpoint. So we must have  ${|f''(e)| \over (f'(e))^2} = 1$.

Since for all $x \in [c_j,d_j]$ one has
${|f''(x)| \over (f'(x))^2} \leq 1$ and $|f'|$ varies by a factor of at most $2$ on an interval of type 2, if $x \in [c_j,d_j]$ we therefore  have
\[ |f''(x)| \leq |f'(x)|^2 \leq 4|f'(e)|^2 = 4|f''(e)| \]
Hence we have 
\[|f''(e)| \geq {1 \over 4} \sup_{ [c_j,d_j]}|f''| \tag{5.15}\]
Next, we use the condition that $\sup_{[a,b]} |f^{(k)}(x)| < A \inf_{[a,b]} |f^{(k)}(x)|$ for some constant $A$, where we recall we are assuming 
$k \geq 2$. First suppose $k \geq 3$. Then we may apply part 2 of Lemma 1 of the famous paper [PhSt] of Phong and Stein to $f''(x)$ and conclude that
 there is a constant $L_{A,k}$ such that on $[c_j,d_j]$ we have the following, where $|I_j|$ denotes the length $d_j - c_j$ of $|I_j|$.
\[ |f'''(x)| \leq L_{A,k} {1 \over |I_j|} \sup_{[c_j,d_j]}|f''| \tag{5.16}\]
In view of $(5.15)$, equation $(5.16)$ implies that there is a constant $M_{A,k}$ such that if $x \in [c_j,d_j]$ satisfies $|x - e| < M_{A,k}|I_j|$, then one has
\[|f''(x)| \geq {1 \over 8} \sup_{[c_j,d_j]}|f''| \tag{5.17}\]
In particular, for such $x$ one has $|f''(x)| \geq {1 \over 8} |f''(e)| = {1 \over 8}|f'(e)|^2 \geq {1 \over 32}|f'(x)|^2$. The final inequality follows from
the fact that $(c_j, d_j)$ is an interval of type 2, so that $|f'|$ varies by a factor of at most 2 on it.
We conclude that for $|x - e| < M_{A,k}|I_j|$ one has
\[{|f''(x)| \over (f'(x))^2} \geq {1 \over 32} \tag{5.18}\]
Now on $I_j$, $\min(1, {|f''(x)| \over (f'(x))^2}) = {|f''(x)| \over (f'(x))^2}$. Thus $(5.18)$ implies that for some constant $N_{A,k} > 0$ one has
\[\int_{I_j} \min\bigg(1, {|f''(x)| \over (f'(x))^2}\bigg)\,dx >N_{A,k}|I_j| \tag{5.19}\]
The above assumed that $k = 3$, but $(5.19)$ will also hold when $k = 2$. For when $k = 2$, not only does $|f'(x)|$ vary by a factor of at most 2 on $I_j$,
but the condition $\sup_{[a,b]} |f''(x)| < A \inf_{[a,b]} |f''(x)|$ implies that $|f''(x)|$ varies by a factor of at most $A$ on $I_j$.  Thus 
${|f''(e)| \over (f'(e))^2} = 1$ implies that ${|f''(x)| \over (f'(x))^2}$ is bounded below by ${1 \over 4A}$ and once again $(5.19)$ will hold.

Next, note that by simply taking absolute values on the inside and integrating, we have
\[\bigg |\int_{c_j}^{d_j} e^{i f(x)}\phi(x)\,dx\bigg| \leq ||\phi||_{L^{\infty}}|I_j|\tag{5.20}\]
Combining this with $(5.19)$ we get that
\[\bigg |\int_{c_j}^{d_j} e^{i f(x)}\phi(x)\,dx\bigg| \leq (N_{A,k})^{-1}||\phi||_{L^{\infty}} \int_{I_j} \min\bigg(1, {|f''(x)| \over (f'(x))^2}\bigg)\,dx 
\tag{5.21}\]
Equation $(5.21)$ is what we need for the portion of the oscillatory integral over $I_j$. 

We now simply add the estimate $(5.2)$ for $G$ to the estimates $(5.12')$ and  $(5.14)$ or $(5.21)$ over all intervals $I_j$ of $H$, and we obtain $(1.3)$.
(One can take the minimum with $b - a$ in the right-hand term of $(1.3)$ since that is given by the bound 
one obtains by simply taking absolute values in the integrand of $\int_a^b e^{i f(x)}\phi(x)\,dx$ and integrating.) This completes 
the proof of Theorem 1.1 for $k > 1$. 

Lastly, we consider the case where $k = 1$. We write $[a,b]$ as the finite union of $j$ intervals $J_i = [c_i,d_i]$, disjoint except at endpoints, such that on 
each $J_i$, $f'(x)$ is monotonic. Then by $k = 1$ case of the Van der Corput lemma $(4.3')$, for some constant $c$ one has
\[\bigg|\int_{c_i}^{d_i}e^{i \lambda f(x)}\phi(x)\,dx\bigg| \leq c (||\phi||_{L^{\infty}} + ||\phi'||_{L^1}) {1 \over \inf_{[c_i,d_i]} |f'(x)|} \tag{5.22}\]
By the condition that $\sup_{[a,b]} |f'(x)| < A \inf_{[a,b]} |f'(x)|$, we see that $(5.22)$ is bounded by
\[ c A (||\phi||_{L^{\infty}} + ||\phi'||_{L^1}) {1 \over \sup_{[a,b]} |f'(x)|} \tag{5.23}\]
Adding this over all $i$ gives the desired inequality
\[\bigg|\int_a^b e^{i \lambda f(x)}\phi(x)\,dx\bigg| \leq  c j A (||\phi||_{L^{\infty}} + ||\phi'||_{L^1}) {1 \over \sup_{[a,b]} |f'(x)|} \tag{5.24}\]
This completes the proof for the $k = 1$ case; one can take the minimum with $b - a$ in the right-hand term of $(1.3)$ exactly as before. This completes the proof of Theorem 1.1.

\subsection{Proof of Theorem 1.2.} 

Suppose the assumptions of Theorem 1.2 hold. We let $\alpha_1 < ... < \alpha_N$ denote the points in $J$, where $J$ is as in Theorem 1.2. We will show 
that there is a constant $B$ such that for each $l < N$ we have
\[\bigg|\int_{\alpha_l}^{\alpha_{l+1}} e^{i f(x)}\phi(x)\,dx\bigg| \leq \]
\[ B (||\phi||_{L^{\infty}} + ||\phi'||_{L^1})\bigg(\int_{\alpha_l}^{\alpha_{l+1}}\min\bigg(1, \bigg|{f''(x) \over (f'(x))^2}\bigg|\bigg) dx + 
{1 \over |f'(\alpha_{l})|} + {1 \over |f'(\alpha_{l+1})|} \bigg) \tag{5.25}\]
Then $(1.6)$ follows by adding $(5.25)$ over all $l$; as in the proof of Theorem 1.1 one can take the minimum with $b - a$ in the right-hand term of $(1.6)$ since that is given by the bound
one obtains by simply taking absolute values in the integrand of $\int_a^b e^{i f(x)}\phi(x)\,dx$ and integrating.

We proceed to the proof of $(5.25)$. We start as in the proof of Theorem 1.1, replacing the interval $[a,b]$ by $[\alpha_l, \alpha_{l+1}]$. Namely we
 let $Z$ be the set
$\{x \in [\alpha_l, \alpha_{l+1}]: x = \alpha_l, x = \alpha_{l+1},$ or $f'(x) = 0\}$. Since $f'''(x) \neq 0$ on $(\alpha_l, \alpha_{l+1})$, the set $Z$ is finite.
Analogous to before, we write $[\alpha_l, \alpha_{l+1}]
= G \cup H$, where $G = \{x \in [\alpha_l, \alpha_{l+1}]: x  \in Z$ or $\big|{f''(x) \over (f'(x))^2}\big| \geq 1\}$, and $H = \{x \in [a,b]: x \notin Z$ and 
$\big|{f''(x) \over (f'(x))^2}\big| < 1\}$. Precisely as in the proof of Theorem 1.1, $(5.2)$ holds, providing the desired estimate for the integral over $G$.

Next, like in the proof of Theorem 1.1 we write $H = \cup_j I_j$, 
where $I_j  = (c_j,d_j)$ are open intervals. Exactly as in the proof of Theorem 1.1 we have that $f'(x) \neq 0$ on each $[c_j,d_j]$. We again divide the 
intervals $I_j$ into type 1 and type 2 intervals, where type 1 intervals are those on which $|f'(x)|$ varies by more than a factor of 2 and type 2 intervals are 
those on which $|f'(x)|$ varies by at most a factor of 2. The analysis of the type 1 intervals is exactly as in the proof of Theorem 1.1, leading to $(5.12')$ 
once again holding.

For type 2 intervals, the argument changes from that of Theorem 1.1, and accounts for the presence of the right-hand terms in $(5.25)$. We further subdivide
the type 2 intervals into two subtypes. We say $I_j$ is an interval of subtype 2A if $c_j = \alpha_l$, $d_j = \alpha_{l+1}$, or both. We say $I_j$ is an interval
of subtype 2B if neither endpoint of $I_j$ is an endpoint of $[\alpha_l, \alpha_{l+1}]$. We first suppose $I_j$ is an interval of type 2A. Thus we may let 
$e$ denote an 
endpoint of $[c_j,d_j]$ that is also an endpoint of $[\alpha_l, \alpha_{l+1}]$.  Since $I_j$ is an interval of type 2, $|f'(x)|$ varies by a factor of at most 2 on $[\alpha_l, \alpha_{l+1}]$. As a result, if $x \in I_j$ we have
\[{1 \over |f'(x)|} \leq {2 \over |f'(e)|} \tag{5.26}\]
We now perform the integration by parts leading to $(5.6)$. We insert $(5.26)$ into $(5.6)$ analogously to the argument leading to $(5.14)$. Analogously
to $(5.14)$, we  obtain
\[\bigg |\int_{c_j}^{d_j} e^{i f(x)}\phi(x) \,dx\bigg| \leq (4||\phi||_{L^{\infty}} + 2||\phi'||_{L^1})\bigg( { 1\over |f'(e)|} + 
\int_{c_j}^{d_j} \min\bigg(1, \bigg|{ f''(x) \over  (f'(x))^2}\bigg|\bigg)\,dx \bigg) \tag{5.27}\] 
This is the estimate we need for the integral over an $I_j$ of type 2A.

We now consider the case where $I_j$ is an interval of type 2B. Since neither $c_j$ nor $d_j$ is an endpoint of $[\alpha_l, \alpha_{l+1}]$ nor a point
where $f' = 0$, by the definition of $H$ one must have that
\[{|f''(c_j)| \over (f'(c_j))^2} = {|f''(d_j)| \over (f'(d_j))^2} = 1 \tag{5.28}\]
The $\alpha_l$ include all points where $f''$ or $f'''$ are zero and $f'$ is nonzero. Because $f'$ is nonzero on $[c_j,d_j]$, this means that
 both $f''$ and $f'''$ are either positive on $(c_j, d_j)$ or negative on $(c_j,d_j)$. Thus
$f'$ and $f''$ are monotonic on $[c_j,d_j]$. Hence there are endpoints $e_1$ and $e_2$ of $[c_j,d_j]$ (which may or may not be the same) such that
$|f'(x)| \leq |f'(e_1)|$ and $|f''(x)| \geq |f''(e_2)|$ for $x \in I_j$. Thus on $I_j$ we have
\[{|f''(x)| \over (f'(x))^2} \geq {|f''(e_2)| \over (f'(e_1))^2} \tag{5.29}\]
But recall $|f'(x)|$ varies by a factor of at most 2 on $I_j$. Thus we have
\[{|f''(e_2)| \over (f'(e_1))^2} \geq {1 \over 4} {|f''(e_2)| \over (f'(e_2))^2}\]
\[= {1 \over 4} \tag{5.30}\]
Thus on $I_j$ we have
\[{|f''(x)| \over (f'(x))^2} \geq {1 \over 4} \tag{5.31}\]
Since  $\min(1, {|f''(x)| \over (f'(x))^2} ) = \big|{f''(x) \over (f'(x))^2}\big|$ on $I_j$, analogously to $(5.19)$ we have
\[\int_{I_j} \min\bigg(1, {|f''(x)| \over (f'(x))^2}\bigg)\,dx >{1 \over 4} |I_j| \tag{5.32}\]
Exactly as in $(5.20)$, by taking absolute values on the inside and integrating one has 
\[\bigg |\int_{c_j}^{d_j} e^{i f(x)}\phi(x)\,dx\bigg| \leq ||\phi||_{L^{\infty}}|I_j|\tag{5.33}\]
Combining $(5.32)$ and $(5.33)$ we thus have
\[\bigg |\int_{c_j}^{d_j} e^{i f(x)}\phi(x)\,dx\bigg| \leq 4 ||\phi||_{L^{\infty}} \int_{I_j} \min\bigg(1, {|f''(x)| \over (f'(x))^2}\bigg)\,dx 
\tag{5.34}\]
Equation $(5.34)$ is the estimate we need for the portion of the oscillatory integral over an $I_j$ of type 2B. We can now add up the estimates for the intervals
of different types, namely $(5.12')$, $(5.27)$, and $(5.34)$, and add the result to the estimate $(5.2)$ for the integral over $G$. The result is $(5.25)$. This
concludes the proof of Theorem 1.2.

\section{Proof of Theorem 2.1.}

Assume that the conditions of Theorem 2.1 are satisfied.
 If $\partial_{x_n}^k g(0) \neq 0$ for some $k$ the theorem is immediate, so we assume that $\partial_{x_n}^k g(0) = 0$
 for each $k$. Equivalently, $g(0,...,0,x_n) = 0$ for each $x_n$. We Taylor expand $g(x)$ on a neighborhood of the origin as
\[ g(x) = \sum_{i=0}^{\infty} g_i(x_1,...,x_{n-1})x_n^i \tag{6.1}\]
Note that each $g_i(x)$ satisfies $g_i(0) = 0$. Thus the ideal $J$ in $\R[[x_1,...,x_{n-1}]]$ generated by all of the $g_i(x)$ is a proper ideal. Note that the ideals $J_i = \langle g_0,...,g_i \rangle$ satisfy the ascending chain condition and their union is $J$. So since 
$\R[[x_1,...,x_{n-1}]]$ is Noetherian, for some $i_0$ one has $J_i = J_{i_0}$ for all $i \geq i_0$ and therefore $J = J_{i_0}$.

We now apply resolution of singularities in $n-1$ dimensions to the nonzero $g_0,...,g_{i_0}$ as well as all nonzero differences $g_i - g_{i'}$ for 
$0 \leq i, i' \leq
i_0$. Hironaka's theorem in [H1] [H2] more than 
suffices for our purposes. As a consequence of these theorems, we may say the following. If $s > 0$ is sufficiently small, then the $n-1$ dimensional closed ball
${\bar B}_{n-1}(0,s)$ can be written as $\cup_{j=1}^N K_j$, where $K_j$ are (overlapping) compact sets containing the origin such that to each 
$K_j$ there is a $U_j$ containing $K_j$ such that the following hold.
\begin{itemize}
\item $\cup_{j=1}^N U_j \subset B_{n-1}(0,s')$ for some $s' > s$ such that $g(x)$ is real analytic on a neighborhood of the closed ball 
$\bar{B}_{n-1}(0,s')$.
\item There is a bounded open $U_j' \subset \R^{n-1}$ containing the origin  and a surjective $\phi_j: U_j' \rightarrow U_j$ whose components are 
real analytic functions 
such that each nonzero $g_i \circ \phi_j(y)$ and each nonzero $(g_i - g_{i'}) \circ \phi_j (y)$ for $0 \leq i, i' \leq i_0$, can be written in the form $a(y)m(y)$
on $U_j'$, where $m(y)$ is a nonconstant monomial and $a(y)$ is a nonvanishing real analytic function.
\item There is a compact $L_j \subset U_j'$ such that $\phi_j(L_j) = K_j$.
\end{itemize}

Next, we define $\bar{\phi}_j(y_1,...,y_n) = (\phi_j(y_1,...,y_{n-1}), y_n)$. 
We examine each $g \circ \bar{\phi}_j(y)$ as a function of $n$ variables on $U_j' \times [-t,t]$, where $t$ is small enough to ensure that each
$g \circ \bar{\phi}_j(y)$ is defined on $U_j' \times [-t,t]$. We may write
\[ g \circ \bar{\phi}_j (y) = \sum_{i=0}^{\infty} g_i \circ \phi_j (y_1,...,y_{n-1})y_n^i \tag{6.2}\]
Let $y_0 = (y_0',0)$ be any point in $L_j \times \{0\}$ such that $\bar{\phi}_j(y_0) = 0$. Equivalently, $\phi(y_0') = 0$. We shift coordinates in the $y$
 variables so that $y_0$ becomes the origin. Namely, we let
 $y  = z+ y_0$. Then  $g_i \circ \phi_j(y_1,...,y_{n-1})$ becomes $g_i \circ \phi_j((z_1,...,z_{n-1})+ y_0')$, which we denote by $h_{ij}(z_1,...,z_{n-1})$.
 We then have
\[ g \circ \bar{\phi}_j(z + y_0) = \sum_{i=0}^{\infty} h_{ij} (z_1,...,z_{n-1})z_n^i \tag{6.3}\]
Because  $\bar{\phi}_j(y_0) = 0$ and each $g_i(0) = 0$, we must have that each $h_{ij}(0) = 0$. Furthermore, each nonzero 
$h_{ij}(z)$ and each nonzero
difference $h_{ij}(z) - h_{i'j}(z)$ for $i, i' \leq i_0$  is still of the form $a(z)m(z)$ on a neighborhood of the origin, where $m(z)$ is nonvanishing and $a(0) \neq 0$. In addition,
since $g_0,...,g_{i_0}$ generate the ideal generated by all $g_i$ on a neighborhood of the origin, given $j$ we also have that $h_{0j},...,h_{i_0 j}$ generate the 
ideal generated by all $h_{ij}$.

Next, we write a nonzero $h_{ij}(z)$ in the form  $a_{ij}(z)m_{ij}(z)$, where $a_{ij}(z)$ is nonvanishing near the origin and $m_{ij}(z)$ is a monomial $z_1^{\alpha_{ij1}}...
z_{n-1}^{\alpha_{ij\,n-1}}$. Because we monomialized the nonzero differences $h_{ij} - h_{i'j}$ for $i, i' \leq i_0$, for a given $j$ we have that for each $i, i' \leq i_0$ we either have
$\alpha_{ij} \leq \alpha_{i'j}$ or $\alpha_{i'j} \leq \alpha_{ij}$. Thus given $j$ there is at least one $i_1 \leq i_0$ such that 
for all $i \leq i_0$ we have $\alpha_{i_1j} \leq \alpha_{ij}$. As a result, $m_{i_1 j}(z)$ divides $m_{ij}(z)$ for all $i \leq i_0$. Thus each $h_{ij}(z)$ 
for $i \leq i_0$ can be written as $m_{i_1 j}(z) q_{ij}(z)$ for a real analytic function $q_{ij}(z)$. Since $h_{0j},...,h_{i_0j}$ generate the ideal generated by all 
$h_{ij}$, for our fixed $j$ each $h_{ij}(z)$ for $i > i_0$ can also be written as $m_{i_1 j}(z) q_{ij}(z)$ for some real analytic function $q_{ij}(z)$. 

In summary, given $j$ one has $h_{i_1j}(z) = a_{i_1j}(z)m_{i_1j}(z)$ where $a_{i_1j}(0) \neq 0$, while for all $i \neq i_1$ we have that 
$h_{ij}(z) =  q_{ij}(z) m_{i_1 j}(z)$ for a real analytic function $q_{ij}(z)$. Thus in view of $(6.3)$, we have
that $\partial_{z_n}^{i_1}  (g \circ \bar{\phi}_j(z + y_0))$ is of the form $r(z_1,...,z_n) m_{i_1 j}(z_1,...,z_{n-1})$ for some real analytic $r(z)$ with $r(0) \neq 0$.
Write $q(z) = g \circ \bar{\phi}_j(z + y_0)$.
Hence if we are in a small enough neighborhood $V$ of the origin so that $|r(z_1,...,z_n)|$ is with a factor of $2$ of $r(0,...,0)$, then for $z \in V$ we 
either have that $q (z)$ is identically zero on the vertical line containing $z$ (corresponding to 
$m_{i_1 j}(z_1,...z_{n-1}) = 0)$, or for $l_j = i_1$ one has
\[ {1 \over 2} |\partial_{z_n}^{l_j} q(z_1,...,z_{n-1},0)| <  |\partial_{z_n}^l q(z_1,...,z_n)| <  2|\partial_{z_n}^{l_j} q(z_1,...,z_{n-1},0)|\tag{6.4} \]
Since $y_0$ was an arbitrary point of $\bar{\phi}_j^{-1}(0) \cap (L_j \times \{0\})$, a compact subset of $L_j \times \{0\}$, we may cover 
$\bar{\phi}_j^{-1}(0)  \cap (L_j \times \{0\})$ by finitely many balls $V_{jl}$ on which $(6.4)$ holds for each $l$ except on vertical lines where 
$q$ is identically zero, and such 
that $\cup_l V_{jl} \subset U_j' \times [-t,t]$ for our small $t$ that ensures $g \circ \bar{\phi}_j(y)$ is well-defined. There is an
open set $W_j$ containing $\phi_j^{-1}(0)  \cap L_j$ and an $\epsilon_j \in (0,t)$ such $W_j \times [-\epsilon_j, \epsilon_j] \subset \cup_l V_{jl}$.
Thus $(6.4)$ holds on $W_j \times [-\epsilon_j, \epsilon_j]$.

If $y$ is in $((W_j)^c \cap L_j) \times [-\epsilon_j, \epsilon_j]$, then $|\bar{\phi}_j(y)| > 0$ since
 ${\phi_j}^{-1}(0) \cap L_j$ is a subset of $W_j$.
 Since $((W_j)^c \cap L_j) \times [-\epsilon_j, \epsilon_j]$ is compact, there is in fact 
a $\delta_j > 0$ such that $|\bar{\phi}_j(y)| \geq \delta_j$ on $((W_j)^c \cap L_j) \times [-\epsilon_j, \epsilon_j]$. As a result, 
$\bar{\phi}_j^{-1}(B(0, \delta_j)) \cap (L_j \times [-\epsilon_j,\epsilon_j]) \subset W_j \times [-\epsilon_j, \epsilon_j]$, a set on which $(6.4)$ holds except on vertical lines on which $q$ is identically zero. Since $(6.4)$ is invariant under coordinate changes in the first $n - 1$ variables such as $\phi_j$, we then
have that the corresponding statement to $(6.4)$ holds on $B(0,\delta_j) \cap (\phi_j(L_j) \times [-\epsilon_j, \epsilon_j])$.  Namely, on this set, on a vertical
line either $g$ is identically zero or one has
\[ {1 \over 2} |\partial_{x_n}^{l_j} g(x_1,...,x_{n-1},0)| <  |\partial_{x_n}^l g(x_1,...,x_n)| <  2|\partial_{x_n}^{l_j}g(x_1,...,x_{n-1},0)|\tag{6.5} \]
Since $\phi_j(L_j) = K_j$, if we let $\eta < \delta_j, \epsilon_j$
for all $j$, then by taking the union of the above over all $j$ we see that on $B(0,\eta) \cap ((\cup_j K_j) \times [-\eta, \eta])$,  on a vertical
line either $g$ is identically zero or one has that $(6.5)$ holds for at least one $l_j$ . Since $\cup_j K_j$ is the original closed ball $\bar{B}_{n-1}(0,s)$ on
which we performed the resolution of singularities, there is an $\eta'$ for which the same is true on
all of $B_{n-1}(0,\eta') \times [-\eta',\eta']$ as is needed.

As for the final statement of Theorem 2.1, as long as $g(x)$ is not identically zero, there is some $m$ such that $\partial_{x_n}^m g(x_1,...,x_{n-1},0)$
is not identically zero on $B_{n-1}(0,\eta')$. Then $\{(x_1,...,x_{n-1}) \in B_{n-1}(0,\eta'): \partial_{x_n}^m g(x_1,...,x_{n-1},0) = 0\}$ has measure zero.
Since the set of $(x_1,...x_{n-1})$ where $g(x_1,...,x_n) = 0$ for all $|x_n| < \eta'$ is a subset of the above set, this set too has measure zero.
This completes the proof of Theorem 2.1.

\section{References}

\noindent [AGuV] V. Arnold, S. Gusein-Zade, A. Varchenko, {\it Singularities of differentiable maps},
Volume II, Birkhauser, Basel, 1988. \parskip = 4pt\baselineskip = 3pt

\noindent [BaGuZhZo] S. Basu, S. Guo, R. Zhang, P. Zorin-Kranich, {\it A stationary set method for estimating oscillatory integrals}, to appear, J. Eur. Math. Soc. 

\noindent [CaCWr] A. Carbery, M. Christ, J. Wright, {\it Multidimensional van der Corput and sublevel set estimates}, J. Amer. Math. Soc. {\bf 12}
 (1999), no. 4, 981-1015. 

\noindent [C] M. Christ, {\it Hilbert transforms along curves. I. Nilpotent groups}, Annals of Mathematics (2) {\bf 122} (1985), no.3, 575-596.

\noindent [ClMi] R. Cluckers and D. Miller, {\it Bounding the decay of oscillatory integrals with a constructible amplitude function and a globally subanalytic phase function}, J. Fourier Anal. Appl. {\bf 22} (2016), no. 1, 215-236. 

\noindent [Gi] M. Gilula, {\it Some oscillatory integral estimates via real analysis}, Math. Z. {\bf 289} (2018), no. 1-2, 377-403.

\noindent [G] J. Green, {\it Lower bounds on $L^p$ quasi-norms and the uniform sublevel set problem}, Mathematika {\bf 67} (2021), no. 2, 296-323.

\noindent [Gr] P. Gressman, {\it Scalar oscillatory integrals in smooth spaces of homogeneous type}, 
Rev. Mat. Iberoam. {\bf 31} (2015), no. 1, 215–244. 

\noindent [H1] H. Hironaka, {\it Resolution of singularities of an algebraic variety over a field of characteristic zero I}, 
 Ann. of Math. (2) {\bf 79} (1964), 109-203.

\noindent [H2] H. Hironaka, {\it Resolution of singularities of an algebraic variety over a field of characteristic zero II},  
Ann. of Math. (2) {\bf 79} (1964), 205-326. 

\noindent [Mi] D. J. Miller, {\it A preparation theorem for Weierstrass systems}, Trans. Amer. Math. Soc. {\bf 358} (2006), no. 10, 4395-4439.

\noindent [PhSt] D. H. Phong, E. M. Stein, {\it The Newton polyhedron and
oscillatory integral operators}, Acta Mathematica {\bf 179} (1997), 107-152.

\noindent [PhStS]  D. H. Phong, E. M. Stein, J. Sturm,  {\it On the growth and stability of real-analytic functions}, Amer. J. Math. {\bf 121} (1999), 
519-554.

\noindent [St] E. M. Stein, {\it Harmonic analysis; real-variable methods, orthogonality, and oscillatory \hfill\break
integrals}, Princeton Mathematics Series {\bf 43}, Princeton University Press, Princeton, NJ, 1993.

\noindent [V] A. N. Varchenko, {\it Newton polyhedra and estimates of oscillatory integrals}, Functional 
Anal. Appl. {\bf 18} (1976), no. 3, 175-196.

\noindent Department of Mathematics, Statistics, and Computer Science \hfill \break
\noindent University of Illinois at Chicago \hfill \break
\noindent 322 Science and Engineering Offices \hfill \break
\noindent 851 S. Morgan Street \hfill \break
\noindent Chicago, IL 60607-7045 \hfill \break
\noindent greenbla@uic.edu

\end{document}